\documentclass[12pt]{article}

\usepackage{amsmath,amssymb,amsthm}
\usepackage{latexsym}
\usepackage{xcolor}
\usepackage{enumitem}
\usepackage{subfigure}
\usepackage{authblk}
\usepackage{a4wide}
\usepackage{epsfig}

\def\RR{\mathbb{R}}

\def\bftau{{\boldsymbol \tau}}

\def\bfxi{{\boldsymbol \xi}}

\def\calP{{\cal P}}

\def\nknots{N}
\def\ndim{\mu}
\def\jmax{l}
\def\nhalf{m}
\def\hmin{h_\mathrm{min}}
\def\per{\mathrm{per}}

\def\pnew{p}
\def\prec{{\bar{p}}}

\newcommand{\ceil}[1]{\lceil#1\rceil}

\newtheorem{theorem}{Theorem}
\newtheorem{lemma}{Lemma}

\newtheorem{corollary}{Corollary}
\newtheorem{conjecture}{Conjecture}

\theoremstyle{remark}\newtheorem{remark}{Remark}
\theoremstyle{remark}\newtheorem{example}{Example}

\usepackage[para]{footmisc}

\newcommand{\refcite}{\cite}

\begin{document}
%
%\markboth{E.~Sande, C.~Manni, and H.~Speleers}
%{Sharp error estimates for spline approximation}
%
%%%%%%%%%%%%%%%%%%%% Publisher's Area please ignore %%%%%%%%%%%%%%%%%%%%%%%
%%
%\catchline{}{}{}{}{}
%%
%%%%%%%%%%%%%%%%%%%%%%%%%%%%%%%%%%%%%%%
%
%\title{\MakeUppercase{Sharp error estimates for spline approximation: explicit constants,} $n$-\MakeUppercase{widths, and eigenfunction convergence} }
%
%\author{\MakeUppercase{Espen Sande}\footnote{Corresponding author}}
%
%\address{Department of Mathematics, University of Oslo \\
%Moltke Moes Vei 35, 0851 Oslo, Norway \\
%espsand@math.uio.no}
%
%\author{\MakeUppercase{Carla Manni}}
%
%\address{Department of Mathematics, University of Rome Tor Vergata \\
%Via della Ricerca Scientifica 1, 00133 Rome, Italy \\
%manni@mat.uniroma2.it}
%
%\author{\MakeUppercase{Hendrik Speleers}}
%
%\address{Department of Mathematics, University of Rome Tor Vergata \\
%Via della Ricerca Scientifica 1, 00133 Rome, Italy \\
%speleers@mat.uniroma2.it}

\title{Sharp error estimates for spline approximation: explicit constants, $n$-widths, and eigenfunction convergence}
%\title{Error estimates}

\author[1]{Espen Sande\thanks{espsand@math.uio.no}}
\author[2]{Carla Manni\thanks{manni@mat.uniroma2.it}}
\author[2]{Hendrik Speleers\thanks{speleers@mat.uniroma2.it}}

\affil[1]{\small Department of Mathematics, University of Oslo, Norway}
\affil[2]{\small Department of Mathematics, University of Rome Tor Vergata, Italy}

\maketitle

%\begin{history}
%\received{(Day Month Year)}
%\revised{(Day Month Year)}
%%\accepted{(Day Month Year)}
%\comby{(xxxxxxxxxx)}
%\end{history}

\begin{abstract}
In this paper we provide a priori error estimates in standard Sobolev (semi-)norms for approximation in spline spaces of maximal smoothness on arbitrary grids. The error estimates are expressed in terms of a power of the maximal grid spacing, an appropriate derivative of the function to be approximated, and an explicit constant which is, in many cases, sharp. Some of these error estimates also hold in proper spline subspaces, which additionally enjoy inverse inequalities. Furthermore, we address spline approximation of eigenfunctions of a large class of differential operators, with a particular focus on the special case of periodic splines. The results of this paper can be used to theoretically explain the benefits of spline approximation under $k$-refinement by isogeometric discretization methods. They also form a theoretical foundation for the outperformance of smooth spline discretizations of eigenvalue problems that has been numerically observed in the literature, and for optimality of geometric multigrid solvers in the isogeometric analysis context.
\end{abstract}

%\keywords{Spline approximation; error estimates; optimal spaces; eigenfunction convergence; inverse inequalities}

%\ccode{AMS Subject Classification: 34L16, 41A15, 41A27, 41A44, 65D07}

%%%%%%%%%%%%%%%%%%%%%%%%%%%%%%%%%%%%%%%%
\section{Introduction}
%%%%%%%%%%%%%%%%%%%%%%%%%%%%%%%%%%%%%%%%
Splines are piecewise polynomial functions that are glued together in a certain smooth way.
When using them in an approximation method, the availability of sharp error estimates is of utmost importance.
Depending on the problem to be addressed, one needs to tailor the norm to measure the error, the properties -- degree and smoothness -- of the approximant, and the space the function to be approximated belongs to. As it is difficult to trace all the works on spline approximation, we refer the reader to \refcite{Schumaker2007} for an extended bibliography on the topic.

The emerging field of isogeometric analysis (IGA) triggered a renewed interest in the topic of spline approximation and related error estimates.
In particular, isogeometric Galerkin methods aim to approximate solutions of variational formulations of differential problems by using spline spaces of possibly high degree and maximal smoothness \cite{Cottrell:09}. In this context, a priori error estimates in Sobolev (semi-)norms and corresponding projectors to a suitably chosen spline space are crucial.

Classical error estimates in Sobolev (semi-)norms for spline approximation are expressed in terms of
\begin{enumerate}
 \item[(a)] a certain power of the maximal grid spacing (this is the approximation power),
 \item[(b)] an appropriate derivative of the function to be approximated, and
 \item[(c)] a ``constant'' which is independent of the previous quantities but usually depends on the spline degree.
\end{enumerate}
An explicit expression of the constant in (c) is not always available in the literature \cite{deBoor2001}, because it is a minor issue in the most standard approximation analysis; the latter is mainly interested in the approximation power of spline spaces of a given degree. 

These estimates are perfectly suited to study approximation under
$h$-refinement, i.e., refining the mesh, and is obtained by the
insertion of new knots; see \refcite{Hughes:18} and references therein.
On the other hand, one of the most interesting features in IGA is $k$-refinement, which denotes degree elevation with increasing interelement smoothness (and requires the use of splines of high degree and smoothness). The above mentioned error estimates are not sufficient to explain the benefits of approximation under $k$-refinement
as long as it is not well understood how the degree of the spline affects the whole estimate, including the ``constant'' in (c).

In this paper we focus on a priori error estimates with explicit constants for approximation by spline functions defined on arbitrary knot sequences. We are able to provide accurate estimates, which are sharp or very close to sharp in several interesting cases. These a priori estimates are actually good enough to cover convergence to eigenfunctions of classical differential operators %in spline spaces of maximal smoothness.
under $k$-refinement.
The key tools to get these results are the theory of Kolmogorov $L^2$ $n$-widths and the representation of the considered Sobolev spaces in terms of integral operators described by suitable kernels \cite{Kolmogorov:36,Pinkus:85}.
The main theoretical contributions and the structure of the paper are outlined in the next subsections.

\subsection{Error estimates}
For $k\geq0$, let $C^k[a,b]$ be the classical space of functions with continuous derivatives of order $0,1,\ldots,k$ on the interval~$[a,b]$.
We further let $C^{-1}[a,b]$ denote the space of bounded, piecewise continuous functions on $[a,b]$ that are discontinuous only at a finite number of points.

Suppose $\bftau := (\tau_0,\ldots,\tau_{\nknots+1})$ is a sequence of (break) points
such that
\begin{equation*}%\label{knots}
a=:\tau_0 < \tau_1 < \cdots < \tau_{\nknots} < \tau_{\nknots+1}:= b,
\end{equation*}
and let
$I_j := [\tau_j,\tau_{j+1})$,
$j=0,1,\ldots,\nknots-1$, and $I_\nknots := [\tau_\nknots,\tau_{\nknots+1}]$.
For any $p \ge 0$, let $\calP_p$ be the space of polynomials of
degree at most $p$. Then, for $-1\leq k\leq p-1$, we define the space $\mathcal{S}^k_{p,\bftau}$ of splines of degree $p$ and smoothness $k$ by
$$ \mathcal{S}^k_{p,\bftau } := \{s \in C^{k}[a,b] : s|_{I_j} \in \calP_p,\, j=0,1,\ldots,\nknots \}, $$
and we set
$$ \mathcal{S}_{p,\bftau} := \mathcal{S}^{p-1}_{p,\bftau}. $$
With a slight misuse of terminology, we will refer to $\bftau$ as knot sequence and to its elements as knots.

For real-valued functions $f$ and $g$ we denote the norm and inner product on $L^2:=L^2(a,b)$ by
$$ \| f\|^2 := (f,f), \quad (f,g) := \int_a^b f(x) g(x) dx, $$
and we consider the Sobolev spaces
$$ H^r:= H^r(a,b)=\{u\in L^2 : \, \partial^\alpha u \in L^2(a,b),\, \alpha=1,\ldots,r\}. $$

Classical results in spline approximation read as follow: for any $u\in H^{r}$ and any $\bftau$ there exists $s_p\in \mathcal{S}^k_{p,\bftau}$
 such that
 \begin{equation}
 \label{eq:classical-err-est}
 \|\partial^{\ell}(u-s_p) \|\leq C(p,k,\ell,r)h^{r-\ell} \| \partial^r u \|, \quad 0\leq \ell \leq r\leq p+1, \quad \ell\leq k+1\leq p,
 \end{equation}
where
\begin{equation}\label{eq:hmax}
h:=\max_{j=0,\ldots,\nknots} h_j, \quad h_j:=\tau_{j+1}-\tau_j.
\end{equation}
The above estimates can be generalized to any $L^q$-norm; see, e.g.,  \refcite{Schumaker2007,Lyche:18}.

A common way to construct a spline approximant is to use a quasi-interpolant, that is a linear combination of the elements of a suitable basis -- usually the B-spline basis -- whose coefficients are obtained by a convenient (local) approximation scheme \cite{deBoorF1973,LycheS1975,Lyche:18}.
Several quasi-interpolants with optimal approximation power are available in the literature; see, e.g., \refcite{Lyche:18} for a constructive example.
When dealing with a specific (B-spline) quasi-interpolant, the constant in \eqref{eq:classical-err-est} often behaves quite badly (exponentially) with respect to the spline degree \cite{Lyche:18}. However, this unpleasant feature is more related to the condition number of the considered basis than to the approximation properties of the spline space. Indeed, it can be proved that the condition number of the B-spline basis in any $L^q$-norm grows like $2^p$ for arbitrary knot sequences \cite{Lyche1978,SchererS1999}.

Mainly motivated by the interest of $k$-refinement in IGA, explicit $p$-dependence in approximation bounds of the form \eqref{eq:classical-err-est} has recently received a renewed attention. In Theorem 2 of \refcite{Buffa:11} a representation in terms of Legendre polynomials has been exploited to provide a constant which behaves like
${1}/(p-k)^{r-\ell} $ for spline spaces of degree $p\geq 2k+1$ and smoothness $k$.
The important case of splines with maximal smoothness has been addressed in \refcite{Takacs:2016}. By considering a proper spline subspace and using Fourier analysis, see Theorems~1.1 and~7.3 of \refcite{Takacs:2016}, it has been proved that for any $u\in H^{r}$ there exists $s_p\in \mathcal{S}_{p,\bftau}$ such that
 \begin{equation}
 \label{eq:Takacs}
 \| u-s_p \|\leq (\sqrt{2}h)^{r} \| \partial^r u \|, \quad 0\leq r\leq p+1,
 \end{equation}
under the assumption of sufficiently fine uniform grids, that is $h_j=h,$ $ j=0,\ldots,\nknots$ and
$hp<b-a$.
The relevance of \eqref{eq:Takacs} is twofold: it covers the case of maximal smoothness and provides a uniform estimate for all the degrees. However, it still suffers from the serious limitation of uniform grid spacing, which is intrinsically related to the use of Fourier analysis. Moreover, it requires a restriction on the grid spacing with respect to the degree.

An interesting framework to examine approximation properties is the theory of Kolmogorov $n$-widths which defines and gives a characterization of optimal $n$-dimensional spaces for approximating function classes and their associated norms \cite{Babuska:2002,Kolmogorov:36,Pinkus:85}.
Kolmogorov $n$-widths and optimal subspaces in $L^2(a,b)$ with respect to the $L^2$-norm were studied in \refcite{Evans:2009} with the goal to (numerically) assess the approximation properties of smooth splines in IGA. 

In a recent sequence of papers, \cite{Floater:2017,Floater:2018,Floater:per}, it has been proved that subspaces of smooth splines of any degree on uniform grids, identified by suitable boundary conditions, are optimal subspaces for $L^2$ Kolmogorov $n$-width problems for certain function classes of importance in IGA and finite element analysis (FEA). The subspaces used in \refcite{Takacs:2016} to prove \eqref{eq:Takacs} are a particular instance of a class of optimal spaces considered in \refcite{Floater:2018}. As a byproduct, the results in \refcite{Takacs:2016} were improved, providing a better constant, in \refcite{Floater:2017} for special sequences $\bftau$ and in \refcite{Floater:2018} for restricted function classes and uniform sequences $\bftau$.
The results in  \refcite{Floater:2017,Floater:2018} were then applied in \refcite{Bressan:preprint} to show that, for uniform sequences $\bftau$ and by comparing the same number of degrees of freedom, $k$-refined spaces in IGA provide better a priori error estimates than $C^0$ FEA and $C^{-1}$ discontinuous Galerkin (DG) spaces in almost all cases of practical interest.

In this paper we complete the extension of the results in \refcite{Takacs:2016}, to arbitrary knot sequences and to any function in $H^r$.
More precisely, we first show the following theorem.

\begin{theorem}\label{thm:one}
For any knot sequence $\bftau$, let $h$ denote its maximum knot distance, and let $P_p$ denote the $L^2(a,b)$-projection onto the spline space $\mathcal{S}_{p,\bftau}$. Then, for any function $u\in H^r(a,b)$,
\begin{equation} \label{eq:Sande}
\|u-P_pu\|\leq \Big(\frac{h}{\pi}\Big)^{r}\|\partial^r u\|,
\end{equation}
for all $p\geq r-1$.
\end{theorem}
We then show that this theorem also holds with $P_p$ replaced by a suitable Ritz projection.
When comparing \eqref{eq:Sande} with \eqref{eq:Takacs} we see that it does not only allow for general knot sequences but also improves on the constant with a factor $(\sqrt{2}\pi)^r$.
Theorem~\ref{thm:one} is a special case of Theorem~\ref{thm:gen}, which additionally provides estimates in higher order semi-norms and for Ritz-type projections.

We further remark that while Theorem~\ref{thm:one} is only stated for the space consisting of maximally smooth splines, $\mathcal{S}_{p,\bftau}$, it also holds for any spline space of lower smoothness, since $\mathcal{S}^{k}_{p,\bftau}\supseteq \mathcal{S}_{p,\bftau}$ for any $k=-1,\ldots,p-1$ and making the space larger does not make the approximation worse. However, it could make the approximation constant better; see Theorem 2 of \refcite{Buffa:11} for cases where $k$ is small enough.

For $r=1$, the error bound in Theorem~\ref{thm:one} still holds if the full spline space is replaced by proper subspaces satisfying certain boundary conditions; see Theorems~\ref{thm:Sp1} and~\ref{thm:reduced}. 
Moreover, any element $s$ in such subspaces satisfies the following inverse inequality
\begin{equation}\label{ineq:inv}
\|s'\|\leq \frac{2\sqrt{3}}{\hmin}\|s\|, 
\end{equation}
where $\hmin$ is the minimum knot distance; see Theorem~\ref{thm:inv}. This generalizes the results in Corollary 5.1 and Theorem 6.1 of \refcite{Takacs:2016} to arbitrary knot sequences and a large class of spline subspaces. They are the main tool for proving optimality of geometric multigrid solvers for linear systems arising from spline discretization methods \cite{Hofreither:2017}.
Note that an extension of \eqref{ineq:inv} to the whole space $\mathcal{S}_{p,\bftau}$ is not possible; see Remark 1.2 of \refcite{Takacs:2016}.

\subsection{Convergence to eigenfunctions}
Spectral analysis can be used to study the error in each eigenvalue and eigenfunction of a numerical discretization of an eigenvalue problem.
For a large class of boundary and initial-value problems the total discretization error on a given mesh can be recovered from its spectral error \cite{Hughes:2008,Hughes:2014}. It is argued in \refcite{Garoni:symbol} that this is of primary interest in engineering applications, since practical computations are not performed in the limit of mesh refinement. Usually the computation is performed on a few, or even just a single mesh, and then the asymptotic information deduced from classical error analysis is insufficient. It is more relevant to understand which eigenvalues/eigenfunctions are well approximated for a given mesh size.
In this paper we use the explicit constant in our a~priori error estimates for the Ritz projections to deduce the spectral error on a given mesh.

As we shall see later, the theory of Kolmogorov $n$-widths and optimal subspaces is closely related to spectral analysis. Assume $A$ is a function class defined in terms of an integral operator $K$. Then, the space spanned by the first $n$ eigenfunctions of the self-adjoint operator $KK^*$ is an optimal subspace for $A$. We show that the general sequence of optimal $n$-dimensional subspaces for $A$, introduced in \refcite{Floater:2018}, then converges to this $n$-dimensional space of eigenfunctions as some parameter $p\to\infty$.
In the most interesting cases, this sequence of optimal $n$-dimensional subspaces consists of spline spaces of degree $p$.
This is naturally connected to a differential operator through the kernel of $KK^*$ being a Green's function. 

By using this general framework, we analyze how well the eigenfunctions of a given differential/integral operator are approximated in optimal subspaces of fixed dimension. In particular, for fixed dimension $n$,
we identify the optimal spline subspaces that converge to spaces spanned by the first $n$ eigenfunctions of the Laplacian subject to different types of boundary conditions, as their degree $p$ increases; see Corollaries~\ref{cor:per} and~\ref{cor:eig}. These results complement those already known in the literature about the convergence of uniform spline spaces to trigonometric functions; see, e.g.,  \refcite{Goodman:83,Ganzburg:2006} and references therein.

We detail and fine-tune our analysis for the relevant case of the eigenfunction approximation of the Laplacian with periodic boundary conditions in the space of periodic splines by the Ritz projection, a projector that can be used to prove convergence of eigenvalues and eigenfunctions of the standard Galerkin eigenvalue problem \cite{Strang:73,Boffi:2010}.
In the case of maximal smoothness $C^{p-1}$ and uniform knot sequence $\bftau$ we consider the periodic $n$-dimensional spline space of degree $p$ and show convergence to the first $n$ or $n-1$ eigenfunctions (depending on the parity of $n$) of the Laplacian with periodic boundary conditions. We conjecture that there is convergence to the first $n$ eigenfunctions for all $n$; see Conjecture~\ref{conj:per} and Remark~\ref{rem:outliers}.

For general smoothness $C^k$, $0\leq k\leq p-1$, and fixed dimension $\ndim$, our error estimate ensures convergence of the projection for increasing $p$ only for a fraction of the eigenfunctions. The amount of this fraction decreases as the maximum knot distance $h$ increases. In particular, if $h=(p-k)/\ndim$ then, roughly speaking, only convergence for the first $\ndim/(p-k)$ of the $\ndim$ considered eigenfunctions is ensured.
It is known that the spectral discretization by B-splines of degree $p$ and smoothness $C^k$ presents $p-k$ branches and only a single branch (the so-called \emph{acoustical} branch) converges to the true spectrum \cite{Hughes:2014,Garoni:symbol}.
This $1/(p-k)$ spectral convergence is in complete agreement with our results; see Remark~\ref{rem:branches} for the details.

\subsection{Outline of the paper}
The remainder of this paper is organized as follows. In Section~\ref{sec:error} we introduce a general framework for obtaining error estimates that we make use of in Section~\ref{sec:spline} to first prove Theorem~\ref{thm:one}, and then to generalize it to higher order semi-norms and to the tensor-product case. This framework is then applied to the periodic case in Section~\ref{sec:per}, where we first obtain an error estimate for periodic splines and then prove convergence, in $p$, to the first eigenfunctions of the Laplacian with periodic boundary conditions. How our error estimates relate to the theory of $n$-widths is explained in Section~\ref{sec:nw}, and their sharpness is discussed in Section~\ref{sec:sharp}. A general convergence result to the eigenfunctions of various differential/integral operators is proved in Section~\ref{sec:eig} and then applied %in Section~\ref{sec:optS} 
to show convergence of certain spline subspaces towards the eigenfunctions of the Laplacian with other boundary conditions.
Section~\ref{sec:reduced} provides error estimates for the class of reduced spline spaces considered in \refcite{Takacs:2016}, and 
inverse inequalities for various spline subspaces are covered in Section~\ref{sec:inverse}.
Finally, we conclude the paper in Section~\ref{sec:conclusions} by summarizing the main theoretical results and some of their practical consequences.

%%%%%%%%%%%%%%%%%%%%%%%%%%%%%%%%%%%%%%%%
\section{General error estimates}\label{sec:error}
%%%%%%%%%%%%%%%%%%%%%%%%%%%%%%%%%%%%%%%%
For $f\in L^2$, let $K$ be the integral operator
$$ K f(x) := \int_a^b K(x,y) f(y) dy. $$
As in \refcite{Pinkus:85}, we use the notation $K(x,y)$ for the kernel of~$K$. We will in this paper only consider kernels that are continuous or piecewise continuous.
We denote by $K^*$ the adjoint, or dual, of the operator $K$,
defined by
$$ (f,K^\ast g) = (Kf, g). $$
The kernel of $K^\ast$ is $K^\ast(x,y) = K(y,x)$.
Similar to matrix multiplication,
the kernel of the composition of two integral operators $K$ and $M$
is
$$ (KM)(x,y) = (K(x,\cdot),M(\cdot,y)). $$

If $\mathcal{X}$ is any finite dimensional subspace of $L^2$ and $P$ denotes the $L^2$-projection onto $\mathcal{X}$, then we are in this paper interested in finding explicit constants $C$ in approximation results of the type
\begin{equation}\label{ineq:u}
\|(I-P)u\|\leq C\|\partial^r u\|,
\end{equation}
that holds for all functions $u$ in some Sobolev space of order $r$. For example, if $r=1$ and the functions $u$ are of the form $u=Kf=\int_a^xf(y)dy$, then \eqref{ineq:u} can equivalently be written as
$$\|(I-P)Kf\|\leq C \|f\|,$$
where $C=\|(I-P)K\|$, the $L^2$-operator norm of $(I-P)K$.

Now, given any finite dimensional subspace $\mathcal{Z}_0\supseteq\mathcal{P}_0$ of $L^2$, and any integral operator $K$, we let $\mathcal{Z}_\prec$ for $\prec\geq 1$ be defined by $\mathcal{Z}_\prec:=\mathcal{P}_0+K(\mathcal{Z}_{\prec-1})$. We further assume that they satisfy the equality
\begin{equation}\label{eq:Xsimpl}
\mathcal{Z}_\prec:=\mathcal{P}_0+K(\mathcal{Z}_{\prec-1}) =\mathcal{P}_0+K^*(\mathcal{Z}_{\prec-1}),
\end{equation}
where the sums do not need to be orthogonal (or even direct).
Moreover, let $P_\prec$ be the $L^2$-projection onto $\mathcal{Z}_\prec$, and define $C\in\RR$ to be
\begin{equation}\label{eq:C}
C:=\max\{\|(I-P_0)K\|,\|(I-P_0)K^*\|\}.
\end{equation}
Observe that if $K$ satisfies $K^*=\pm K$, then $C=\|(I-P_0)K\|$ and \eqref{eq:Xsimpl} is true for any initial space $\mathcal{Z}_0$. An integral operator satisfying $K^*=- K$ will be considered in Section~\ref{sec:per}.

Using the argument of Lemma~1 in \refcite{Floater:2017}, we obtain the following result.
\begin{lemma}\label{lem:1simpl}
If the spaces $\mathcal{Z}_\prec$ satisfy \eqref{eq:Xsimpl} and $P_\prec$ denotes the $L^2$-projection onto $\mathcal{Z}_\prec$, then
\begin{equation*}%\label{ineq:1simpl}
\|(I-P_\prec)K\|\leq \|K-KP_{\prec-1}\|\leq C, \quad \forall \prec\geq 1.
\end{equation*}
\end{lemma}
\begin{proof}
We see from \eqref{eq:Xsimpl} that $KP_{\prec-1}$ maps into the space $\mathcal{Z}_\prec$ for $\prec\geq 1$. Now, since $P_\prec$ is the best approximation into $\mathcal{Z}_\prec$ we have,
\begin{align*}
\|(I-P_\prec)K\|\leq \|K-KP_{\prec-1}\|=\|(I-P_{\prec-1})K^*\|.
\end{align*}
Continuing this procedure gives
\begin{align*}
\|(I-P_\prec)K\|\leq\begin{cases}
  \|(I-P_0)K\|, & \prec \text{ even},\\
   \|(I-P_0)K^*\|, & \prec \text{ odd},
\end{cases}
\end{align*}
and the result follows from the definition of $C$ in \eqref{eq:C}.
\end{proof}

We now generalize the above result using an argument similar to Lemma~1 in \refcite{Floater:2018}.
\begin{lemma}\label{lem:2simpl}
Let $r\geq 1$ be given.
If the spaces $\mathcal{Z}_\prec$ satisfy \eqref{eq:Xsimpl} and $P_\prec$ denotes the $L^2$-projection onto $\mathcal{Z}_\prec$, then
\begin{align*}
\|(I-P_\prec)K^r\|\leq \|K^r-KP_{\prec-1}K^{r-1}\|\leq C\,\|(I-P_{\prec-1})K^{r-1}\|,
\end{align*}
for all $\prec\geq 1$.
\end{lemma}
\begin{proof}
Similar to the previous lemma, $KP_{\prec-1}K^{r-1}f\in \mathcal{Z}_\prec$ is some approximation to $K^rf$ and so, since $P_\prec K^rf$ is the best approximation,
\begin{align*}
\|K^r-P_\prec K^r\| &\leq \|K^r-KP_{\prec-1}K^{r-1}\| =\|K(I-P_{\prec-1})K^{r-1}\|\\
&\leq \|K(I-P_{\prec-1})\|\,\|(I-P_{\prec-1})K^{r-1}\|\\
&=\|(I-P_{\prec-1})K^*\|\,\|(I-P_{\prec-1})K^{r-1}\|,
\end{align*}
where we used $(I-P_{\prec-1})=(I-P_{\prec-1})^2$.
The result now follows from Lemma~\ref{lem:1simpl} since $\|(I-P_{\prec-1})K^*\|\leq C$ for all $\prec\geq 1$.
\end{proof}

Similar to Theorem~4 in \refcite{Floater:2018}, we obtain the following result.
\begin{theorem}\label{thm:simple}
If the spaces $\mathcal{Z}_\prec$ satisfy \eqref{eq:Xsimpl} and $P_\prec$ denotes the $L^2$-projection onto $\mathcal{Z}_\prec$, then
\begin{align*}%\label{ineq:L2proj}
\|(I-P_\prec)K^r\|\leq \|K^r-KP_{\prec-1}K^{r-1}\|\leq C^r, %\|(I-P_0)K\|^r,
\end{align*}
for all $\prec\geq r-1$.
\end{theorem}
\begin{proof}
The case $r=1$ is Lemma \ref{lem:1simpl}. The cases $r\geq 2$ then follow from Lemma~\ref{lem:2simpl} and induction on $r$.
\end{proof}

%%%%%%%%%%%%%%%%%%%%%%%%%%%%%%%%%%%%%
\section{Spline approximation}\label{sec:spline}
%%%%%%%%%%%%%%%%%%%%%%%%%%%%%%%%%%%%%
In this section we prove, and generalize, Theorem \ref{thm:one}.
Consider the integral operator $K$ defined by integrating from the left,
\begin{equation}\label{eq:Kint}
(Kf)(x):=\int_a^xf(y)dy.
\end{equation}
One can check that the operator $K^*$ is then integration from the right,
\begin{equation*}
(K^*f)(x)=\int_x^bf(y)dy;
\end{equation*}
see, e.g., Section~7 of \refcite{Floater:2018}.
The space $H^r$ can then be given as
\begin{equation}\label{eq:Hr}
H^r=\mathcal{P}_{0} + K(H^{r-1})=\mathcal{P}_{0} + K^*(H^{r-1})
=\mathcal{P}_{r-1}+K^r(H^0),
\end{equation}
with $H^0=L^2$, and the spline spaces $\mathcal{S}_{p,\bftau}$ satisfy
\begin{equation}\label{eq:Sp}
\mathcal{S}_{p,\bftau} = \mathcal{P}_0+K(\mathcal{S}_{p-1,\bftau}) = \mathcal{P}_0+K^*(\mathcal{S}_{p-1,\bftau}).
\end{equation}
Note that neither of the sums in \eqref{eq:Hr} and \eqref{eq:Sp} are orthogonal.
Next, recall the following Poincar\'e inequality (see, e.g., \refcite{Payne:60}): for any $u\in H^1$ on the interval $(a,b)$ we have 
\begin{equation}\label{ineq:Poinc}
\|u-\bar{u}\|\le \frac{b-a}{\pi}\|u'\|,
\end{equation}
where ${\bar u}:=(b-a)^{-1}\int_a^b u(x)dx$ is the mean value of $u$. This result can be proved using Fourier analysis and it is also the case $n=1$ in \refcite{Kolmogorov:36}.
Let $P_0$ be the $L^2$-projection onto $\mathcal{S}_{0,\bftau}$ and $\|\cdot\|_{j}$ be the $L^2$-norm on the knot interval $I_j$. Then, using the Poincar\'e inequality on each knot interval, 
we have for all $u\in H^1$ that
\begin{equation}\label{ineq:deg0proof}
\|u-P_0u\|^2= \sum_{j=0}^\nknots\|u-P_0u\|_j^2\le \sum_{j=0}^\nknots \Big(\frac{h_j}{\pi}\Big)^2\|u'\|^2_j.
\end{equation}
In combination with \eqref{eq:hmax}, we obtain
\begin{equation}\label{ineq:deg0}
\|u-P_0u\|\le \frac{h}{\pi}\|u'\|.
\end{equation}
Since $K(L^2), K^*(L^2)\subset H^1$ for $K$ in \eqref{eq:Kint}, it follows that for $\mathcal{Z}_0=\mathcal{S}_{0,\bftau}$ the constant $C$ in \eqref{eq:C} satisfies
\begin{equation}\label{ineq:deg0-C}
C\leq h/\pi.
\end{equation}
Using Theorem~\ref{thm:simple} we can now prove Theorem~\ref{thm:one}.

\begin{proof}[Proof of Theorem \ref{thm:one}]
Recall that $P_p$ denotes the $L^2$-projection onto $\mathcal{S}_{p,\bftau}$, and observe that $u=g+K^rf\in H^r$ for $f\in L^2$ and $g\in\mathcal{P}_{r-1}\subset \mathcal{S}_{p,\bftau}$. Then, using \eqref{eq:C} with \eqref{ineq:deg0-C} in Theorem~\ref{thm:simple} (with $\prec=p$) we arrive at
\begin{equation*}
\|u-P_pu\|=\|(g+K^rf)-P_p(g+K^rf)\|\leq \|(I-P_p)K^r\|\,\|f\|\leq \Big(\frac{h}{\pi}\Big)^r\|\partial^ru\|,
\end{equation*}
for all $p\geq r-1$.
\end{proof}

\subsection{Higher order semi-norms}\label{subsec:firstgen}
We now generalize Theorem \ref{thm:one} to higher order semi-norms.
To do this, we define a sequence of projection operators $Q_p^q:H^q\to \mathcal{S}_{p,\bftau}$, for $q=0,\ldots,p$, by $Q_p^0:=P_p$ and
\begin{equation}\label{eq:Qproj}
(Q_p^qu)(x) := c(u) + (KQ_{p-1}^{q-1}\partial u)(x) = c(u) + \int_a^x(Q_{p-1}^{q-1}\partial u)(y) d y,
\end{equation}
where $c(u)\in\mathcal{P}_0$ is chosen such that
\begin{equation}\label{eq:Qconst}
\int_a^b(Q_p^qu)(x) d x=\int_a^b u(x)dx.
\end{equation}
Observe that these projections, by definition, commute with the derivative: $\partial Q_p^q=Q_{p-1}^{q-1}\partial$. Note also that the range of $Q_p^q$ is $ \mathcal{S}_{p,\bftau}$ for any $q=0,\ldots,p$, since the spline spaces themselves satisfy $\partial^q \mathcal{S}_{p,\bftau} = \mathcal{S}_{p-q,\bftau}$ for any $q=0,\ldots,p$.

\begin{theorem}\label{thm:gen}
Let $u\in H^r$ for $r\geq 1$ be given.
For any $q=1,\ldots,r-1$ and knot sequence $\bftau$, let $Q_p^q$ be the projection onto $S_{p,\bftau}$ defined in \eqref{eq:Qproj}. Then,
\begin{align}
\|\partial^{q-1}(u-Q^q_pu)\|&\leq \Big(\frac{h}{\pi}\Big)^{r-q+1}\|\partial^ru\|,
\label{ineq:semi1}\\
\|\partial^q(u-Q^q_pu)\|&\leq \Big(\frac{h}{\pi}\Big)^{r-q}\|\partial^ru\|,
\label{ineq:semi2}
\end{align}
for all $p\geq r-1$.
\end{theorem}
\begin{proof}
From \eqref{eq:Hr} we know that $u\in H^r$ can be written as $u=g+K^{q}v$ for $g\in\mathcal{P}_{q-1}\subset\mathcal{S}_{p,\bftau}$ and $v\in H^{r-q}$.
Then, by using the fact that $\partial^qQ^q_p=Q^{q-q}_{p-q}\partial^q=P_{p-q}\partial^q$, inequality \eqref{ineq:semi2} immediately follows from Theorem~\ref{thm:one}:
\begin{equation*}
\|\partial^q(u-Q_p^qu)\|
=\|v-P_{p-q}v\|\leq \Big(\frac{h}{\pi}\Big)^{r-q}\|\partial^{r-q}v\|= \Big(\frac{h}{\pi}\Big)^{r-q}\|\partial^ru\|, \quad p\geq r-1.
\end{equation*}

Next, we look at inequality \eqref{ineq:semi1}. First observe that with $u$ as above we have
\begin{align*}
\|\partial^{q-1}(u-Q^q_pu)\|&=\|\partial^{q-1}(g+K^qv-Q^q_p(g+K^qv))\|=\inf_{c\in\mathcal{P}_0}\|Kv-c-KP_{p-q}v\|,
\end{align*}
where we used the commuting property $\partial^{q-1}Q^q_p=Q^{1}_{p-q+1}\partial^{q-1}$ together with the definition in \eqref{eq:Qproj} and \eqref{eq:Qconst}. The above infimum is taken over all $c\in\mathcal{P}_0$, and so by making the choice $c=0$ we obtain
\begin{align*}
\|\partial^{q-1}(u-Q^q_pu)\| \leq \|Kv-KP_{p-q}v\|=\|K(I-P_{p-q})v\|.
\end{align*}
The function $v\in H^{r-q}$ can be written as $v=\hat g+K^{r-q}f$ for $f\in L^2$ and $\hat g\in\mathcal{P}_{r-q-1}\subset \mathcal{S}_{p-q,\bftau}$, and so
$$\|K(I-P_{p-q})v\|= \|(K^{r-q+1}-KP_{p-q}K^{r-q})f\|.$$
Inequality \eqref{ineq:semi1} now follows from Theorem \ref{thm:simple} (with $\mathcal{Z}_0=\mathcal{S}_{0,\bftau}$ and $\prec=p-q+1$) and \eqref{ineq:deg0-C}.
\end{proof}
\begin{remark}
%We remark that the 
The above proof of inequality \eqref{ineq:semi1} can also be used to obtain an error estimate in the case $q=r$. Specifically, we have
\begin{equation}
\|\partial^{r-1}(u-Q^r_pu)\|\leq \frac{h}{\pi}\|\partial^ru\|,\quad\forall p\geq r,
\end{equation}
where the extra requirement on the degree, $p\geq r$, is needed to ensure that the projection $Q^r_p$ (or equivalently $P_{p-r}$) is well-defined. By using $\partial^rQ^r_p=P_{p-r}\partial^r$, one can also obtain the stability estimate $\|\partial^{r}(u-Q^r_pu)\|\leq \|\partial^ru\|$ for $p\geq r$.
\end{remark}

\begin{example}\label{ex:H1}
Let $q=1$. Since $\partial(\mathcal{S}_{p,\bftau})=\mathcal{S}_{p-1,\bftau}$, the projection operator $Q_p^1$ can equivalently be defined as the solution to the Neumann problem: find $Q_p^1u\in\mathcal{S}_{p,\bftau}$ such that
\begin{equation*}
\begin{aligned}
(\partial Q_p^1u,\partial v) &= (\partial u,\partial v),\quad \forall v\in \mathcal{S}_{p,\bftau},\\
(Q_p^1u,1)&=(u,1),
\end{aligned}
\end{equation*}
and this projection is usually referred to as a Ritz projection or a Rayleigh-Ritz projection.
Theorem \ref{thm:gen} then states that this approximation of $u\in H^r$, $r\geq 2$, satisfies the error estimates
\begin{equation}\label{ineq:RitzH1}
\begin{aligned}
\|u-Q^1_pu\|&\leq \Big(\frac{h}{\pi}\Big)^{r}\|\partial^ru\|,  &&\forall p\geq r-1,\\
\|\partial(u-Q^1_pu)\|&\leq \Big(\frac{h}{\pi}\Big)^{r-1}\|\partial^ru\|,  &&\forall p\geq r-1.
\end{aligned}
\end{equation}
Thus $Q_p^1u$ provides a good approximation of both the function $u$ itself, and its first derivative.
\end{example}

%%%%%%%%%%%%%%%%%%%%%%%%%%%%%%%%%%%%%%%%%
\subsection{Extension to higher dimensions}\label{subsec:tens}
%%%%%%%%%%%%%%%%%%%%%%%%%%%%%%%%%%%%%%%%%
In this subsection we briefly mention how to extend the error estimate in Theorem~\ref{thm:one} to the tensor-product case.
Let $\Omega:=(a_1,b_1)\times(a_2,b_2)$ and let $\|\cdot\|_{\Omega}$ denote the $L^2(\Omega)$-norm.
The following corollary can be concluded from Theorem \ref{thm:one} with the aid of Theorem~8 in \refcite{Bressan:preprint}, but for the sake of completeness we also provide a short proof here.
\begin{corollary}
Let $P_{p_1,p_2}:=P_{p_1}\otimes P_{p_2}$ be the $L^2(\Omega)$-projection onto $\mathcal{S}_{p_1,\bftau_1}\otimes \mathcal{S}_{p_2,\bftau_2}$, and 
let $h:=\max\{h_{\bftau_1},h_{\bftau_2}\}$ where $h_{\bftau_i}$ denotes the maximum knot distance in $\bftau_i$, $i=1,2$. 
Then, for any $u\in H^{r}(\Omega)$ we have
\begin{align*}
\|u-P_{p_1,p_2}u\|_{\Omega}\leq \Big(\frac{h}{\pi}\Big)^r\Big(\|\partial_x^{r}u\|_{\Omega}+\|\partial_y^{r}u\|_\Omega\Big),
\end{align*}
for all $p_1,p_2\geq r-1$.
\end{corollary}
\begin{proof}
From the triangle inequality and the fact that $P_{p_1}\otimes P_{p_2}=P_{p_1}\circ P_{p_2}$, we obtain
\begin{align*}
\|u-P_{p_1}\otimes P_{p_2}u\|_{\Omega}&\leq \|u-P_{p_1}u\|_{\Omega}+\|P_{p_1}u-P_{p_1}\circ P_{p_2}u\|_{\Omega}\\
&\leq \|u-P_{p_1}u\|_{\Omega}+\|P_{p_1}\|\,\|u-P_{p_2}u\|_{\Omega}\\
&\leq\Big(\frac{h}{\pi}\Big)^r\Big(\|\partial_x^{r}u\|_{\Omega}+\|\partial_y^{r}u\|_\Omega\Big),
\end{align*}
where we used Theorem~\ref{thm:one} in each direction, together with the fact that the $L^2(\Omega)$-operator norm of $P_{p_1}$ is equal to $1$.
\end{proof}

%%%%%%%%%%%%%%%%%%%%%%%%%%%%%%%%%%%%%%%%%%
\section{Results for periodic spline spaces}\label{sec:per}
%%%%%%%%%%%%%%%%%%%%%%%%%%%%%%%%%%%%%%%%%%
In this section we consider the Sobolev space of periodic functions,
$$ H^r_{\per}:=\{u\in H^r: \, \partial^\alpha u(0)=\partial^\alpha u(1),\, \alpha=0,1,\ldots,r-1\}, $$
and the periodic spline space $\mathcal{S}_{p,\bftau,\per}$ defined by
$$ \mathcal{S}_{p,\bftau,\per} := \{s\in \mathcal{S}_{p,\bftau}:\, \partial^\alpha s(0)=\partial^\alpha s(1),\,\alpha=0,1,\ldots,p-1\}. $$
We remark that we only consider the interval $(a,b)=(0,1)$ to simplify the exposition below.
Note that $\mathcal{S}_{0,\bftau,\per}=\mathcal{S}_{0,\bftau}$.
Later (in Section~\ref{subsec:pereig}) we will make use of the dimension of $\mathcal{S}_{p,\bftau,\per}$ and so in this section we index the break points in $\bftau$ such that $n=\dim \mathcal{S}_{p,\bftau,\per}$, i.e.,
%$\bftau=(\tau_1,\ldots,\tau_{n-1})$.
$\bftau=(\tau_0,\ldots,\tau_n)$.

Now, let $K$ be the integral operator of \refcite{Floater:per}.
On the interval $(0,1)$ its kernel has the explicit representation
\begin{equation}\label{eq:Kper}
K(x,y)=\begin{cases}
  -x+y-1/2, & x<y,\\
  -x+y+1/2, & x\geq y.
  \end{cases}
\end{equation}
Using this kernel one can check that the integral operator $K$ satisfies $K^*=-K$.
If $f\perp 1$ then $K(x,y)$ is the Green's function to the boundary value problem (see Lemma~1 in \refcite{Floater:per})
\begin{equation}\label{bvp:per}
u'(x)=f(x),\quad x\in (0,1),\quad u(0)=u(1),\quad u\perp 1,
\end{equation}
meaning that $u$ is the solution of \eqref{bvp:per} if and only if it satisfies $u=Kf$. We note that if $f$ is not orthogonal to $1$ then $Kf=K(f-\int_0^1 f(x)dx)$.
By using \eqref{bvp:per} it was shown in \refcite{Floater:per} that the space $H^r_{\per}$ is equal to
\begin{equation}\label{eq:Hper}
H^r_{\per} = \mathcal{P}_0\oplus K(H^{r-1}_{\per})
\end{equation}
with $H^0_\per=L^2$, and the spline spaces $\mathcal{S}_{p,\bftau,\per}$ satisfy
\begin{equation}\label{eq:Sper}
\mathcal{S}_{p,\bftau,\per}=\mathcal{P}_0\oplus K(\mathcal{S}_{p-1,\bftau,\per}).
\end{equation}
The sums in \eqref{eq:Hper} and \eqref{eq:Sper} are orthogonal.
Again, for $q=0,\ldots,p$, we can define a sequence of projection operators $Q_p^q: H^q_{\per}\to \mathcal{S}_{p,\bftau,\per}$, in exactly the analogous way to Section~\ref{subsec:firstgen}, by $Q_p^0$ being the $L^2$-projection and
\begin{equation}\label{eq:Qprojper}
(Q_p^qu)(x) := c(u) + (KQ_{p-1}^{q-1}\partial u)(x),
\end{equation}
where $K$ now has kernel \eqref{eq:Kper}, and
where $c(u)\in\mathcal{P}_0$ is chosen such that
\begin{equation*}
\int_0^1(Q_p^qu)(x) d x=\int_0^1 u(x)dx.
\end{equation*}
Just as before, these projections commute with the derivative, $\partial Q_p^q=Q_{p-1}^{q-1}\partial$.
Now, using \eqref{eq:Hper} and \eqref{eq:Sper}, together with the fact that $\partial^q(\mathcal{S}_{p,\bftau,\per})=\mathcal{S}_{p-q,\bftau,\per}$,
one can check that $Q_p^q$ is a Ritz projection and solves the problem
\begin{equation} \label{eq:biharmonic}
\begin{aligned}
(\partial^qQ_p^qu,\partial^qv)&=(\partial^qu,\partial^qv), \quad \forall v\in \mathcal{S}_{p,\bftau,\per},\\
(Q_p^qu,1)&=(u,1),
\end{aligned}
\end{equation}
for all $0\leq q\leq p$.

\subsection{Error estimates}
In many applications one would be interested in finding a single spline function that can provide a good approximation of all derivatives of $u$ up to a given number $q$. The next theorem shows that $Q_p^qu$ is such a spline function, for $p$ large enough.

\begin{theorem}\label{thm:per}
Let $u\in H^r_{\per}$ for $r\geq 1$ be given.
For any $q=0,\ldots,r-1$ and knot sequence $\bftau$, let $Q_p^q$ be the projection onto $S_{p,\bftau,\per}$ defined in \eqref{eq:Qprojper}. Then, for any $\ell=0,\ldots,q$ we have
\begin{align} \label{ineq:per}
\|\partial^{\ell}(u-Q^q_pu)\|&\leq \Big(\frac{h}{\pi}\Big)^{r-\ell}\|\partial^ru\|,
\end{align}
for all $p$ satisfying both $p\geq r-1$ and $p\geq 2q-\ell-1$.
\end{theorem}
\begin{proof}
From \eqref{eq:Hper} we know that $u\in H_{\per}^r$ can be written as the orthogonal sum $u=c+K^{r}f$ for $c\in\mathcal{P}_{0}$, $f\in L^2$ and $K$ in \eqref{eq:Kper}. Thus,
\begin{align*}
\|\partial^{\ell}(u-Q^q_pu)\|&=\|K^{r-\ell}f-Q_{p-\ell}^{q-\ell}K^{r-\ell}f\|
 =\|K^{r-\ell}f-K^{q-\ell}P_{p-q}K^{r-q}f\| \\
&\leq \|K^{q-\ell}(I-P_{p-q})K^{r-q}\|\,\|f\| \\
&\leq \|K^{q-\ell}(I-P_{p-q})\|\,\|(I-P_{p-q})K^{r-q}\|\,\|\partial^r u\|,
\end{align*}
where we used that $(I-P_{p-q})=(I-P_{p-q})^2$.
Using Theorem \ref{thm:simple} and the Poincar\'e inequality \eqref{ineq:deg0}, now applied to functions in $H^1_\per\subset H^1$, we then find that
\begin{align*}
\|(I-P_{p-q})K^{r-q}\|&\leq \Big(\frac{h}{\pi}\Big)^{r-q},  &&\forall p\geq r-1,\\
\|K^{q-\ell}(I-P_{p-q})\|&=\|(I-P_{p-q})(K^{q-\ell})^*\|\leq \Big(\frac{h}{\pi}\Big)^{q-\ell}, &&\forall p\geq 2q-\ell-1.
\end{align*}
\end{proof}

We remark that the case $\ell=0$ in the above theorem improves upon the constant in \refcite{Takacs:2016} for uniform knot sequences and generalizes the approximation results for periodic splines in  \refcite{Floater:per,Pinkus:85,Takacs:2016} to an arbitrary knot sequence $\bftau$.

\begin{example}
Similar to Example \ref{ex:H1}, let $q=1$ and $r\geq 2$.
Theorem \ref{thm:per} then states that the above approximation $Q_p^1u$ of $u\in H^r_{\per}$ satisfies the error estimates
\begin{equation*} %\label{ineq:H1per}
\begin{aligned}
\|u-Q^1_pu\|&\leq \Big(\frac{h}{\pi}\Big)^{r}\|\partial^ru\|,  &&\forall p\geq r-1,\\
\|\partial(u-Q^1_pu)\|&\leq \Big(\frac{h}{\pi}\Big)^{r-1}\|\partial^ru\|,  &&\forall p\geq r-1.
\end{aligned}
\end{equation*}
Thus $Q_p^1u$ provides a good approximation of both the function $u$ itself, and its first derivative.
\end{example}

\begin{example}
Let $q=2$ and $r=3$. For $Q_p^2u$ to approximate $u\in H^3_{\per}$ in the $L^2$-norm, the above theorem requires the degree to be at least $2q-1=3$, and not $r-1=2$ as one might expect. In view of \eqref{eq:biharmonic}, this is consistent with the known fact that the biharmonic equation must be solved with piecewise polynomials of at least cubic degree to obtain an optimal rate of convergence in $L^2$; see, e.g., p.~118 in \refcite{Strang:73}.
\end{example}

%%%%%%%%%%%%%%%%%%%%%%%%%%%%%%%%%%%%%%%%
\subsection{Convergence to eigenfunctions}\label{subsec:pereig}
%%%%%%%%%%%%%%%%%%%%%%%%%%%%%%%%%%%%%%%%
Consider the periodic eigenvalue problem
\begin{equation}\label{eq:per}
-u''(x) = \nu u(x), \quad x\in (0,1), \quad u(0)=u(1), \quad u'(0)=u'(1).
\end{equation}
It has eigenvalues given by $\nu_0=0$ and
\begin{equation}\label{eq:eigvper}
\nu_{2i-1}=\nu_{2i}=(2\pi i)^2,\quad i=1,2,\ldots,
\end{equation}
with corresponding orthonormal eigenfunctions $\psi_0=1$ and
\begin{equation}\label{eq:eigper}
\psi_j=\sqrt{2}\begin{cases}\sin(2\pi i x), &j=2i-1,\\ \cos(2\pi i x), \quad &j=2i,\end{cases} \quad j=1,2,\ldots.
\end{equation}
Since $\psi_j\in H^r_{\per}$ for any $r$, we can plug these eigenfunctions into estimate \eqref{ineq:per} and obtain the following result.

\begin{corollary}\label{cor:per}
Let $q\geq \ell\geq 0$ be given and let $Q_p^q$ be the projection onto $\mathcal{S}_{p,\bftau,\per}$ defined in \eqref{eq:Qprojper}.
Then, for all $j$ satisfying $2\ceil{j/2}h<1$, we have
\begin{equation}\label{ineq:eigper}
\begin{aligned}
\|\partial^\ell(\psi_j-Q_p^q\psi_j)\| &\leq (2\ceil{{j}/{2}} h)^{p+1-\ell} \xrightarrow[p\to\infty]{} 0.
\end{aligned}
\end{equation}
\end{corollary}
\begin{proof}
First note that $\psi_0=1\in \mathcal{S}_{p,\bftau,\per}$ for all $p\geq 0$, and so $\psi_0-Q_p^q\psi_0=0$. From \eqref{eq:eigper} we have
$$\|\partial^r\psi_j\| = \begin{cases}(2\pi i)^r, &j=2i-1\\ (2\pi i)^r, &j=2i\end{cases},\quad j=1,2,\ldots,$$
which can be written more compactly as $\|\partial^r\psi_j\|=(2\pi\ceil{j/2})^r$. Using Theorem~\ref{thm:per} with $r=p+1$ we then obtain \eqref{ineq:eigper}.
\end{proof}

Let $[\ldots]$ denote the span of a set of functions. Then, any $v\in[\psi_0,\ldots,\psi_{\jmax}]$ can be written as $v(x)=\sum_{j=0}^{\jmax}c_j\psi_j(x)$. So if we let $\jmax$ be the largest $j$ satisfying $2\ceil{j/2}h<1$, then from Corollary~\ref{cor:per} we have
\begin{align*}
\frac{\|v-Q_p^qv\|^2}{\|v\|^2}&=\frac{\|\sum_{j=0}^{\jmax}c_j(\psi_j-Q_p^q\psi_j)\|^2}{\|\sum_{j=0}^{\jmax}c_j\psi_j\|^2}
\leq \frac{(\sum_{j=0}^{\jmax}|c_j|\,\|\psi_j-Q_p^q\psi_j\|)^2}{\sum_{j=0}^{\jmax}c_j^2} \\
&\leq \sum_{j=0}^{\jmax}\|\psi_j-Q_p^q\psi_j\|^2
\leq (\jmax+1)\Big(2\ceil{\jmax/2}h\Big)^{2(p+1)} \xrightarrow[p\to\infty]{} 0.
\end{align*}
Thus, the $n$-dimensional spline space $\mathcal{S}_{p,\bftau,\per}$ approximates the whole $(\jmax+1)$-dimensional space $[\psi_0,\ldots,\psi_{\jmax}]$ as $p\to\infty$.

\begin{remark}\label{rem:perGal}
The error estimate in Corollary~\ref{cor:per} can be used to prove convergence, in $p$, to eigenvalues and eigenfunctions of the standard Galerkin eigenvalue problem: find $\psi^h_j\in \mathcal{S}_{p,\bftau,\per}$ and $\nu_j^h\in\RR$, $j=0,1,\ldots,n-1$, such that
\begin{align}\label{eq:per-discr}
(\partial \psi^h_j,\partial \phi) = \nu_j^h(\psi^h_j,\phi),\quad \forall \phi\in\mathcal{S}_{p,\bftau,\per}.
\end{align}
In Chapter~6.3 of \refcite{Strang:73} (see also Part~2.8 of \refcite{Boffi:2010}) it is shown that the error in the eigenvalues, $|\nu_j-\nu_j^h|$, and the error in the eigenfunctions, $\|\psi_j-\psi_j^h\|$, can be bounded in terms of $\|\psi_j-Q_p^1\psi_j\|$, which goes to $0$ as $p\to\infty$ for all $j$ satisfying the requirement of Corollary~\ref{cor:per} (with $q=1$). The argument can also be extended to periodic eigenvalue problems of higher order ($q$-harmonic with $q>1$).
\end{remark}

\begin{example}\label{ex:per}
Let $\bftau$ be the uniform knot sequence. Then, $h=1/n$ and if
\begin{itemize}
\item $n$ is an odd number, say $n=2\nhalf-1$, then
$$2\ceil{j/2}h=2\ceil{j/2}/(2\nhalf-1)<1,\quad j=0,\ldots,2\nhalf-2,$$
and so, as $p\to\infty$, the $(2\nhalf-1)$-dimensional spline space $\mathcal{S}_{p,\bftau,\per}$ approximates the whole $(2\nhalf-1)$-dimensional space of eigenfunctions,
\begin{align}\label{pereig}
[1,\sin(2\pi x),\cos(2\pi x),\ldots,\sin(2\pi(\nhalf-1) x),\cos(2\pi(\nhalf-1) x)].
\end{align}
\item $n$ is an even number, say $n=2\nhalf$, then
$$2\ceil{j/2}h=\ceil{j/2}/\nhalf<1,\quad j=0,\ldots,2\nhalf-2,$$
and so, as $p\to\infty$, the $2\nhalf$-dimensional spline space $\mathcal{S}_{p,\bftau,\per}$ also approximates the $(2\nhalf-1)$-dimensional space of eigenfunctions in \eqref{pereig}. Note that for $j=2\nhalf-1$ we have $\ceil{j/2}/\nhalf=1$ and so the a priori estimate in Theorem~\ref{thm:per} does not imply convergence to the last eigenfunction in this case. This is reasonable since both $\sin(2\pi \nhalf x)$ and $\cos(2\pi \nhalf x)$ (and any linear combination of them) is a ``candidate'' for being the $2\nhalf$-th eigenfunction.
\end{itemize}
\end{example}

In the above example we observed that if $\bftau$ is a uniform knot sequence and if the dimension of $\mathcal{S}_{p,\bftau,\per}$ is equal to $2\nhalf$, then the a priori estimate in Theorem~\ref{thm:per} only guarantees convergence to the first $2\nhalf-1$ periodic eigenfunctions in \eqref{pereig}. One can check that in this case the piecewise constant spline space $\mathcal{S}_{0,\bftau,\per}$ is orthogonal to $\cos(2\pi \nhalf x)$, since $\cos(2\pi \nhalf x)$ integrates to $0$ on each knot interval. Using integration by parts one can then find that $\mathcal{S}_{1,\bftau,\per}$ is orthogonal to $\sin(2\pi \nhalf x)$, and in general, that the even-degree spline spaces are orthogonal to $\cos(2\pi \nhalf x)$ and the odd-degree spline spaces are orthogonal to $\sin(2\pi \nhalf x)$. We therefore make the following conjecture.

\begin{conjecture}\label{conj:per}
Let $\bftau$ be the uniform knot sequence such that the dimension of $\mathcal{S}_{p,\bftau,\per}$ is equal to $2\nhalf$.
For any $q\geq 0$ let $Q_p^q$ be the projection onto $\mathcal{S}_{p,\bftau,\per}$ defined in \eqref{eq:Qprojper}. We then conjecture that
\begin{equation*}
\begin{aligned}
\|\sin(2\pi \nhalf\cdot)-Q^q_{2i}\sin(2\pi \nhalf\cdot)\| &\xrightarrow[i\to\infty]{} 0, &&p=2i,\\
\|\cos(2\pi \nhalf\cdot)-Q^q_{2i+1}\cos(2\pi \nhalf\cdot)\| &\xrightarrow[i\to\infty]{} 0,  &&p=2i+1.
\end{aligned}
\end{equation*}
\end{conjecture}
In Section~\ref{sec:optS} we look at the Laplacian with different (non-periodic) boundary conditions and find $n$-dimensional spline spaces where a corresponding error estimate guarantees convergence, in $p$, to the $n$ first eigenfunctions for all $n$ (and not just for $n$ odd/even).

\begin{remark}\label{rem:outliers}
The periodic eigenvalue problem \eqref{eq:per} could also be discretized with a Galerkin method as in \eqref{eq:per-discr} using the larger spline space
\begin{equation}\label{eq:larger-per-space}
 \{s\in \mathcal{S}_{p,\bftau}:\, \partial^\alpha s(0)=\partial^\alpha s(1), \, \alpha=0,\ldots,k\},
\end{equation}
for some $0\leq k\leq p-1$ and uniform knot sequence $\bftau=(\tau_0,\ldots,\tau_n)$.
With such a discretization, a very poor approximation of the largest $p-k-1$ eigenvalues is observed numerically for all tested $n$; see Figure~\ref{fig:outliers} for some examples. These are usually referred to as \emph{outlier} modes \cite{Hughes:2014}.
Note that the number of outlier modes depends on the kind of boundary conditions of the eigenvalue problem to be solved (see \refcite{Hughes:2014} for homogeneous Dirichlet boundary conditions).
Since $\mathcal{S}_{p,\bftau,\per}$ is a subspace of \eqref{eq:larger-per-space}, it follows from Remark~\ref{rem:perGal} and Example~\ref{ex:per} that we have convergence of the Galerkin eigenvalue approximation in the space \eqref{eq:larger-per-space} for the first $n$ or $n-1$ eigenvalues according to the parity of $n$. If Conjecture~\ref{conj:per} is true, this number can be raised to $n$ in all cases. 
This is in agreement with the number of outlier modes observed numerically, because the dimension of the space \eqref{eq:larger-per-space} is $n+p-k-1$.
\end{remark}

\begin{figure}
\centering
\subfigure[$p=3$, $k=0$]{\includegraphics[trim=15 5 15 0,width=0.45\textwidth]{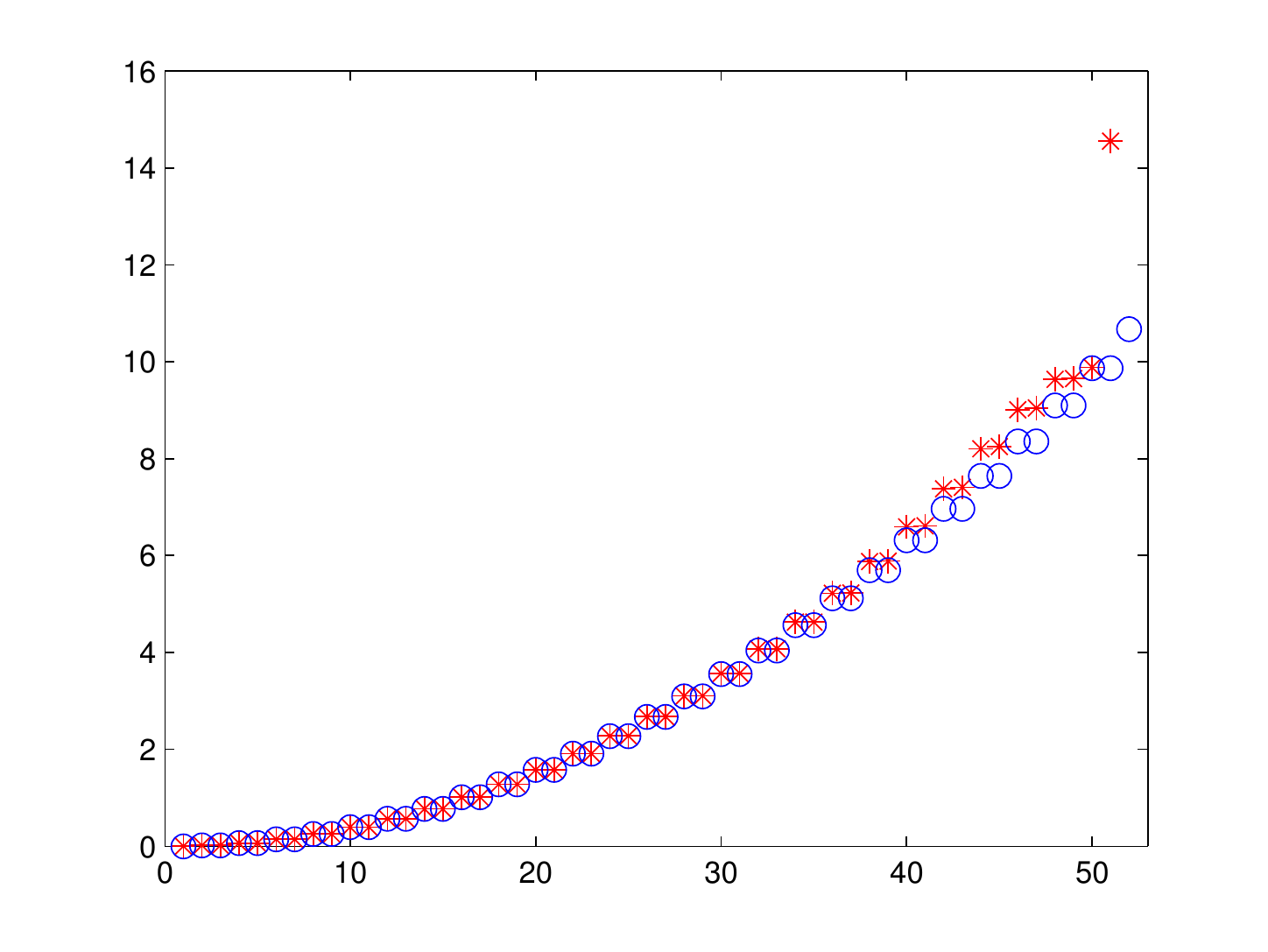}}
\subfigure[$p=6$, $k=0$]{\includegraphics[trim=15 5 15 0,width=0.45\textwidth]{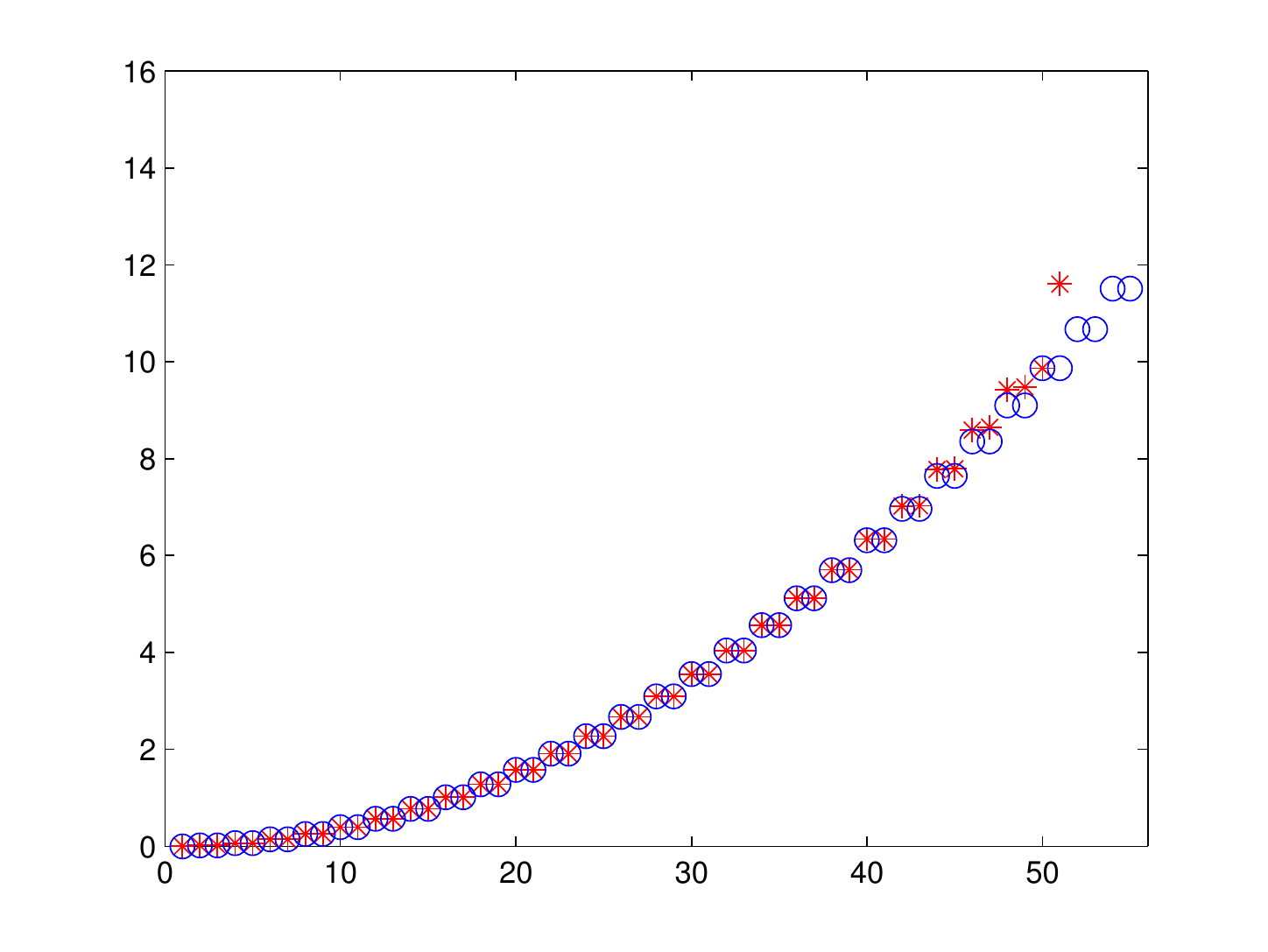}} \\
\subfigure[$p=3$, $k=1$]{\includegraphics[trim=15 5 15 0,width=0.45\textwidth]{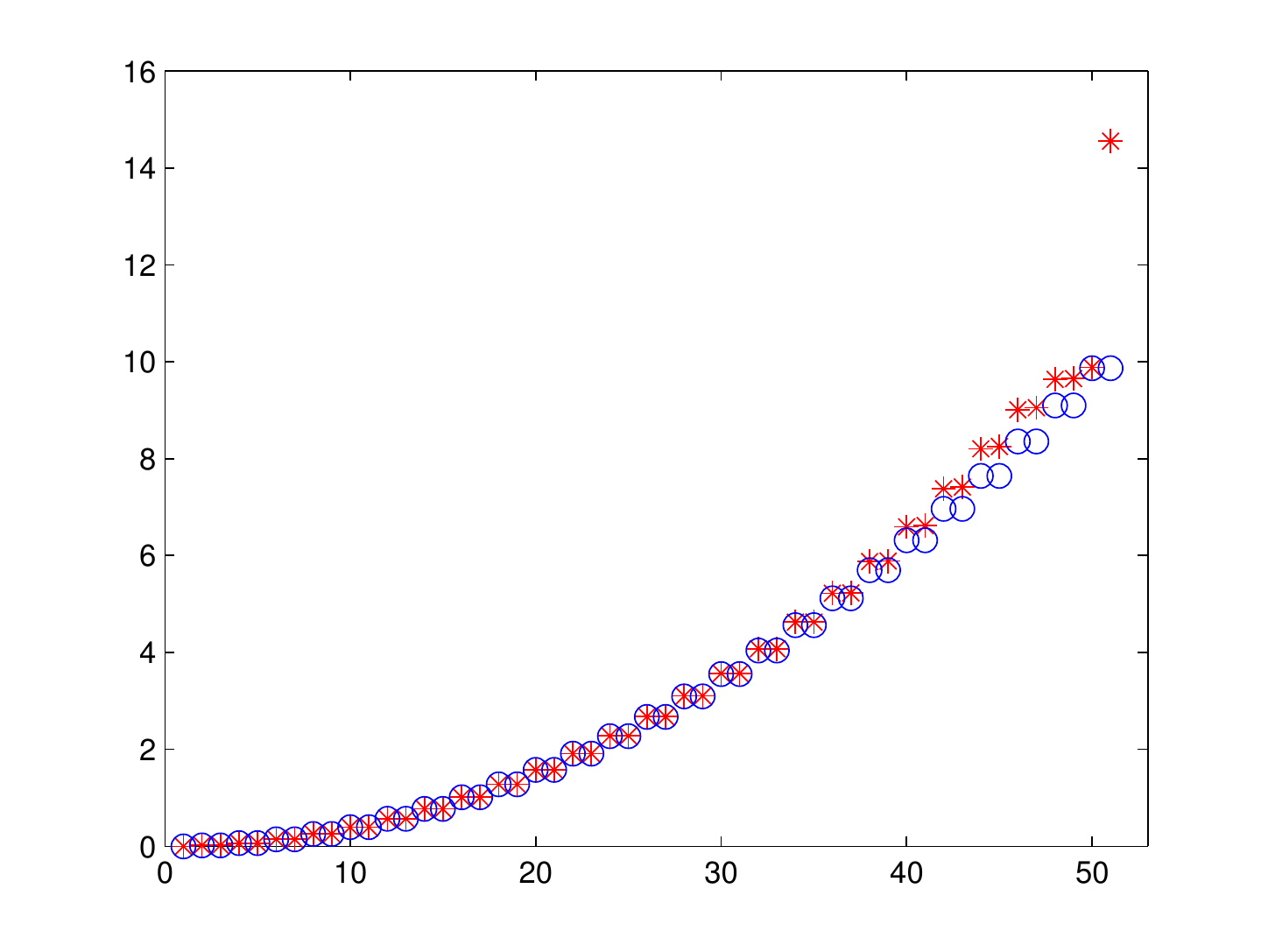}}
\subfigure[$p=6$, $k=4$]{\includegraphics[trim=15 5 15 0,width=0.45\textwidth]{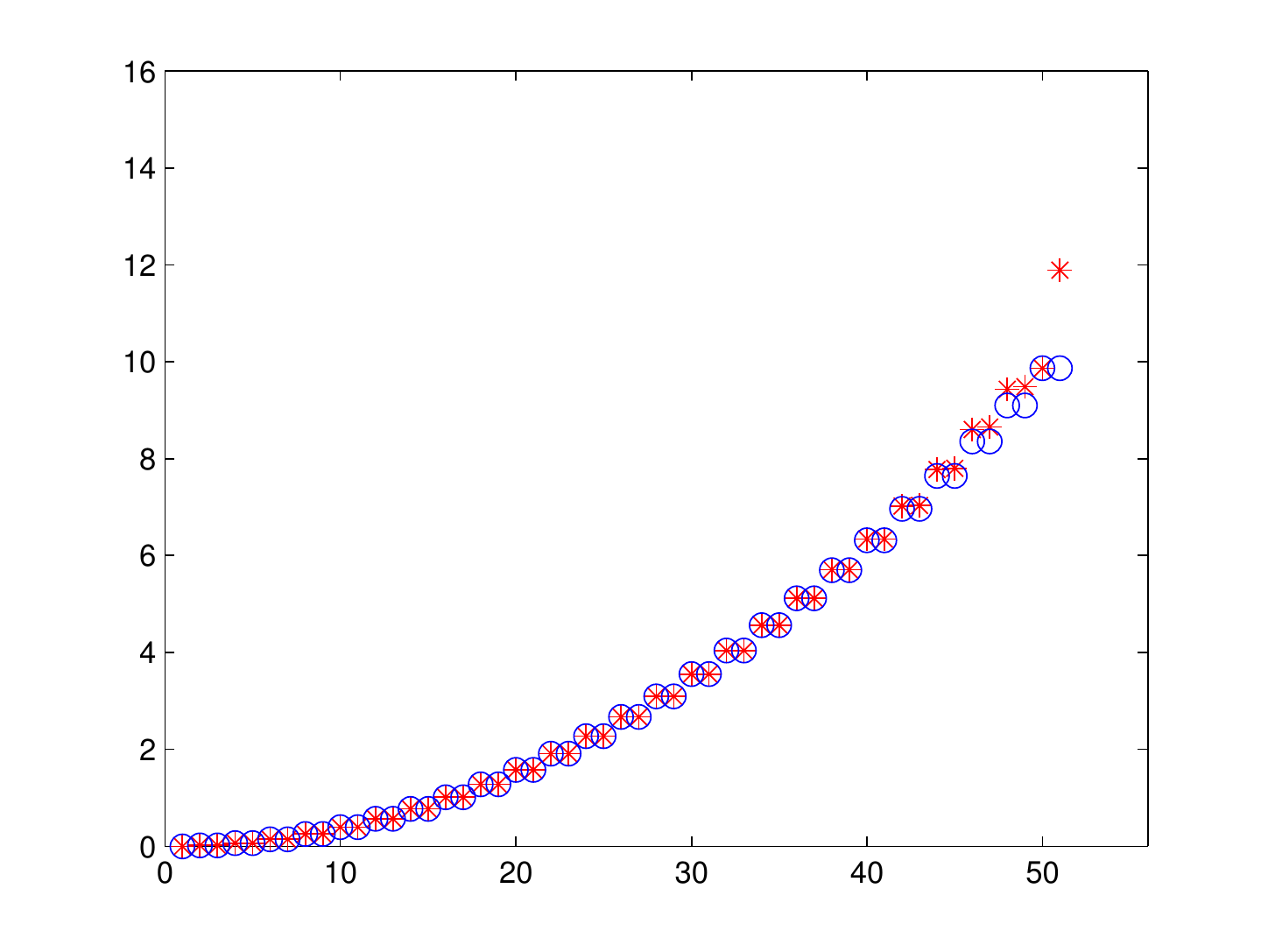}} \\
\subfigure[$p=3$, $k=2$]{\includegraphics[trim=15 5 15 0,width=0.45\textwidth]{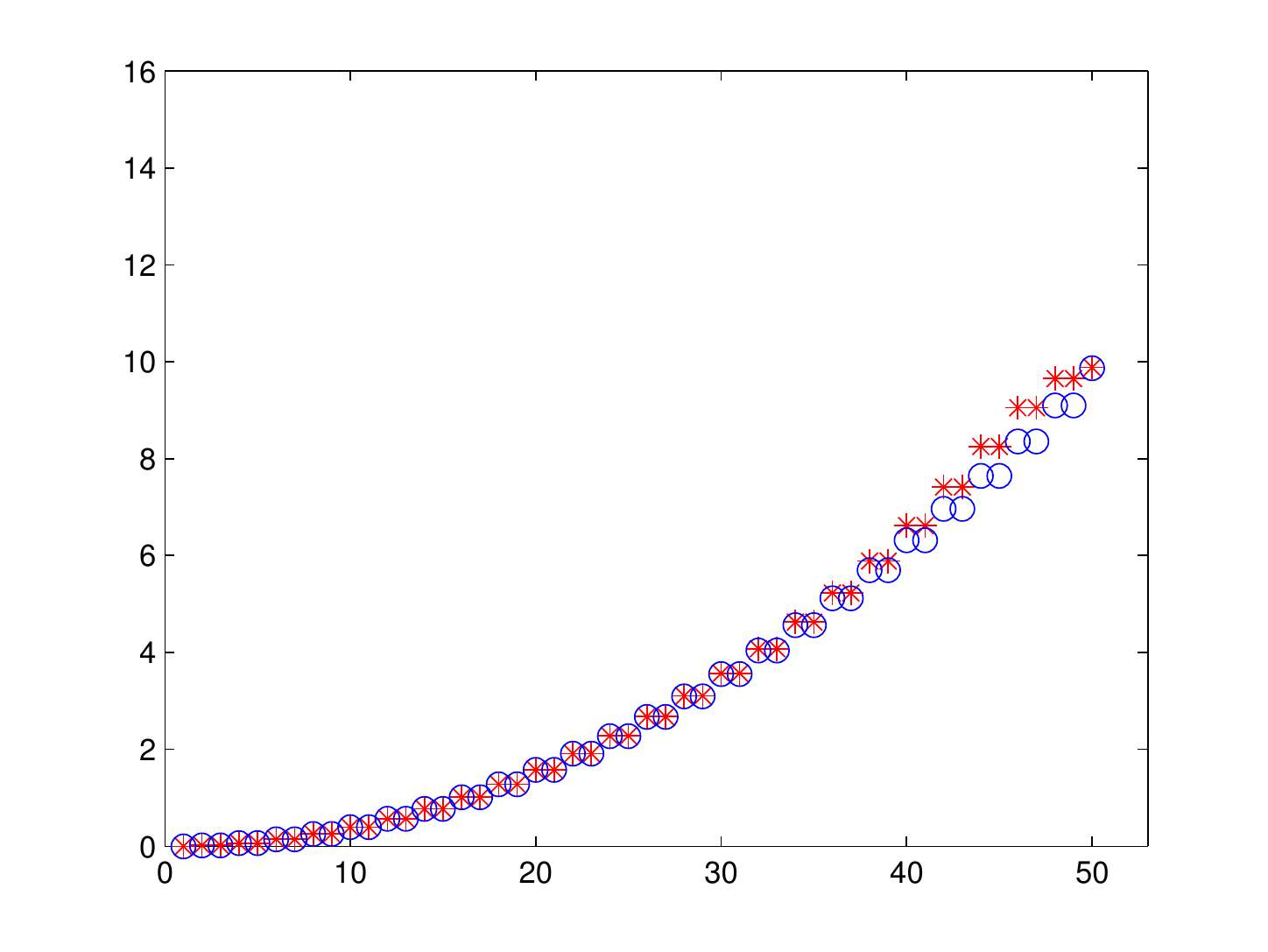}}
\subfigure[$p=6$, $k=5$]{\includegraphics[trim=15 5 15 0,width=0.45\textwidth]{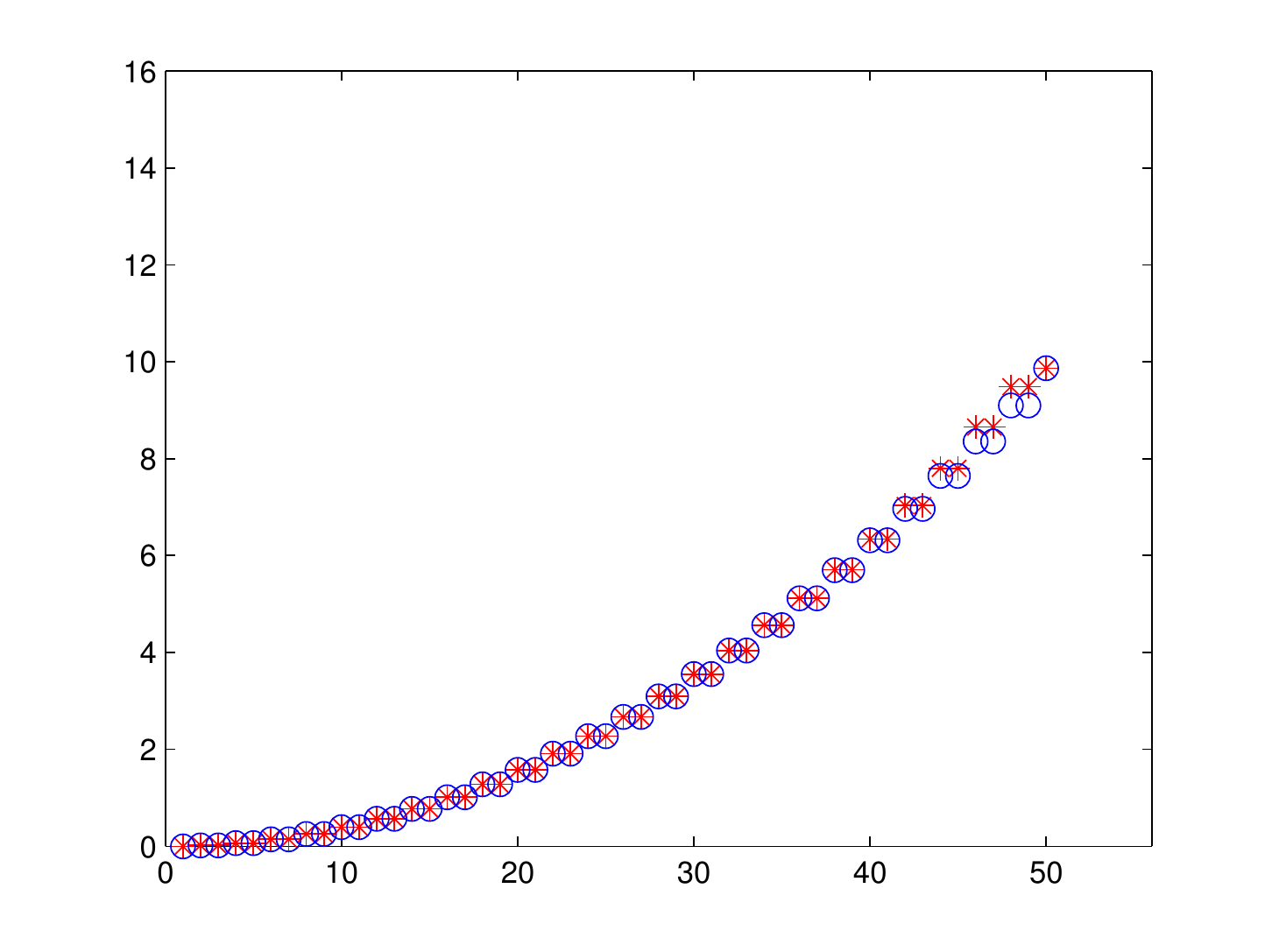}}
\caption{Discretization of the periodic eigenvalue problem \eqref{eq:per} in the space \eqref{eq:larger-per-space} with fixed $n=50$ and varying $p\in\{3,6\}$, $k\in\{0,p-2,p-1\}$: exact eigenvalues (in {blue $\circ$}) and approximated ones (in {red $*$}).
One clearly observes $p-k-1$ outlier modes. Note that in the two top pictures the last $p-2$ approximated eigenvalues are not shown because their value exceeds the adopted scale.
} \label{fig:outliers}
\end{figure}

\begin{remark}\label{rem:branches}
Let us now consider the spline space $\mathcal{S}^k_{p,\bftau}$ for $0\leq k\leq p-1$ and uniform knot sequence $\bftau=(\tau_0,\ldots,\tau_n)$. One can then discretize problem \eqref{eq:per} with a Galerkin method using the $C^k$ periodic spline space
\begin{equation}\label{eq:even-larger-per-space}
\{s\in \mathcal{S}^k_{p,\bftau}:\, \partial^\alpha s(0)=\partial^\alpha s(1), \, \alpha=0,\ldots,k\}.
\end{equation}
The dimension of this space is $n(p-k)$.
Hence, by increasing the degree of this space one substantially increases its dimension (when $k$ is small). However, numerical evidence shows that the spectral discretization of the Laplacian by splines of degree~$p$, smoothness $C^k$, and uniform grid spacing, possesses $p-k$ branches of equal length and only a single branch (the so-called \emph{acoustical} branch \cite{Hughes:2014,Garoni:symbol}) converges to the true spectrum; see Figure~\ref{fig:branches} for some examples.
Since our results can only guarantee convergence to eigenvalues in this acoustical branch, they are in complete agreement with the numerical evidence.
\end{remark}

\begin{figure}[t!]
\centering
\subfigure[$p=3$]{\includegraphics[trim=15 5 15 0,width=0.49\textwidth]{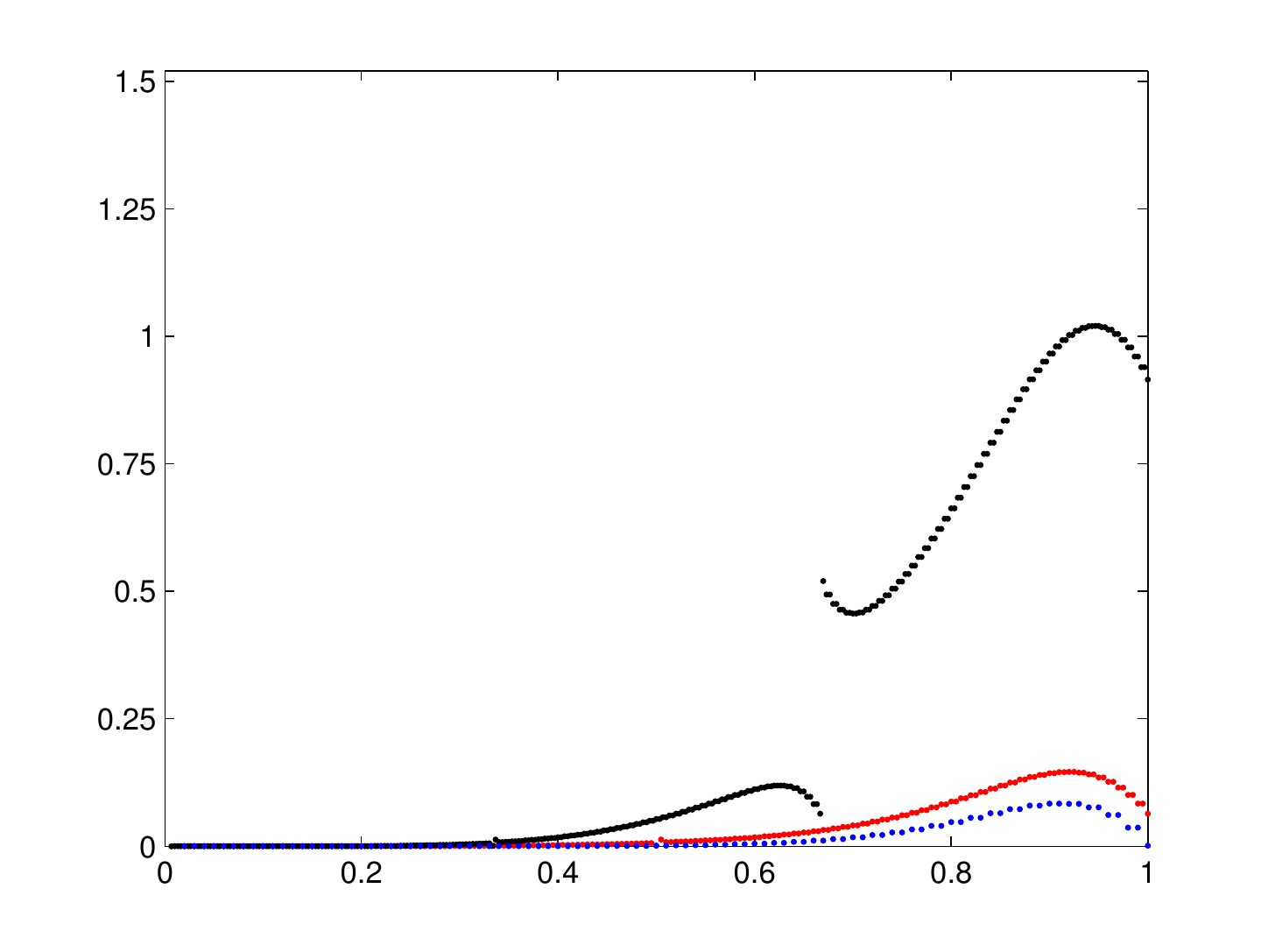}}
\subfigure[$p=4$]{\includegraphics[trim=15 5 15 0,width=0.49\textwidth]{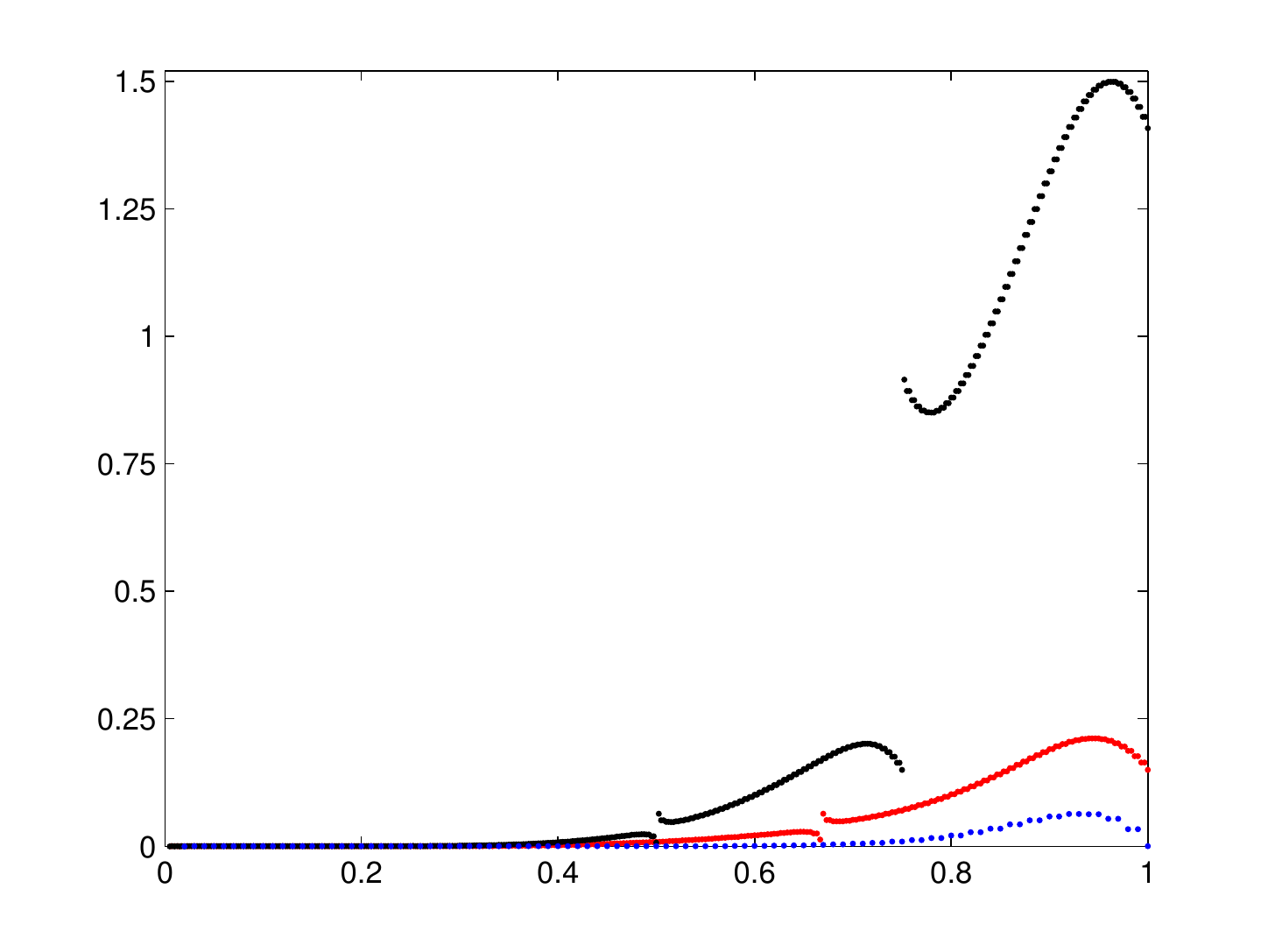}} \\
\subfigure[$p=3$, magnified view]{\includegraphics[trim=15 5 15 0,width=0.49\textwidth]{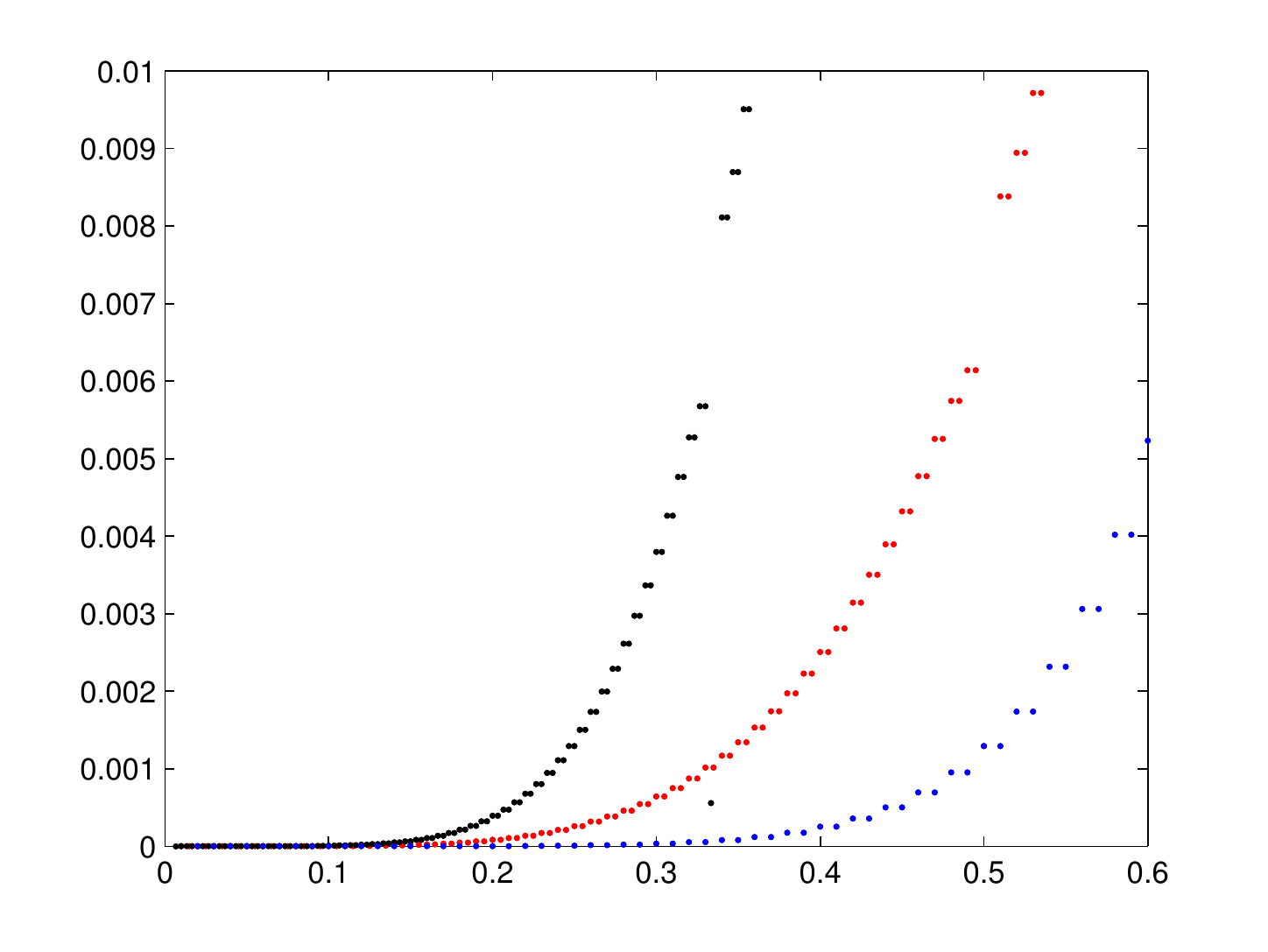}}
\subfigure[$p=4$, magnified view]{\includegraphics[trim=15 5 15 0,width=0.49\textwidth]{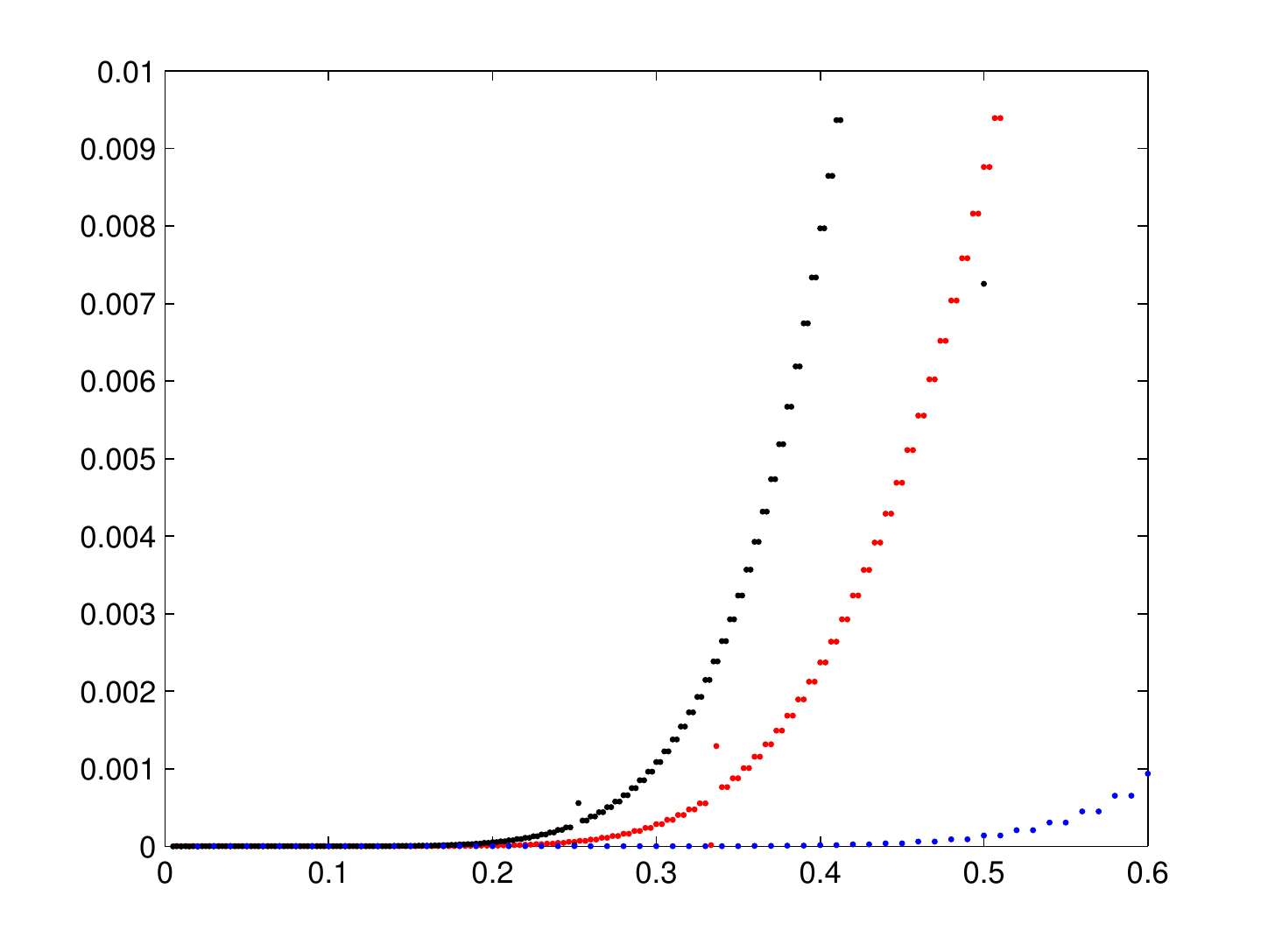}}
\caption{Discretization of the periodic eigenvalue problem \eqref{eq:per} in the space \eqref{eq:even-larger-per-space} with fixed $n=100$ and varying $k$ ($k=0$ in black, {$k=1$ in red}, {$k=p-1$ in blue}): relative eigenvalue approximation error $\nu^h_j/\nu_j-1$, $j=2,\ldots,n(p-k)$, where each $\nu^h_j$ denotes the approximated value of the $j$-th eigenvalue $\nu_j$. All cases are plotted in the interval $[0,1]$, after a proper rescaling, as it is common in the literature. One clearly observes $p-k$ spectral branches of equal length.
} \label{fig:branches}
\end{figure}

%%%%%%%%%%%%%%%%%%%%%%%%%%%%%%%%%%%%%%%
\section{$n$-Widths and kernels}\label{sec:nw}
%%%%%%%%%%%%%%%%%%%%%%%%%%%%%%%%%%%%%%%%
Our next goal is to discuss the sharpness of our error estimates. To do this, we first introduce the theory of $n$-widths \cite{Kolmogorov:36,Pinkus:85}.
As before, we denote by $P$ the $L^2$-projection onto a finite dimensional subspace $\mathcal{X}$ of $L^2$.
For a subset $A$ of $L^2$, let
$$ E(A, \mathcal{X}) := \sup_{u \in A} \|u-Pu\| $$
be the distance to $A$ from $\mathcal{X}$ relative to the $L^2$-norm.
Then, the Kolmogorov $L^2$ $n$-width
of $A$ is defined by
$$ d_n(A) := \inf_{\substack{\mathcal{X}\subset L^2\\ \dim \mathcal{X}=n}} E(A, \mathcal{X}). $$
If $\mathcal{X}$ has dimension at most $n$ and satisfies
\begin{equation*}%\label{eq:optimal}
d_n(A) = E(A, \mathcal{X}),
\end{equation*}
then we call $\mathcal{X}$ an \emph{optimal} subspace for $d_n(A)$.

\begin{example}
Let $A=\{u\in H^r : \|\partial^r u\|\leq 1\}$.
Then, by looking at $u/\|\partial^r u\|$, for functions $u\in H^r$, we have for any subspace $\mathcal{X}$ of $L^2$, the sharp estimate
\begin{equation}\label{ineq:sharp}
\|u-Pu\|\leq E(A, \mathcal{X})\|\partial^r u\|.
\end{equation}
Here $E(A, \mathcal{X})$ is the least possible constant for the subspace $\mathcal{X}$.
Moreover, if $\mathcal{X}$ is optimal for the $n$-width of $A$, then
\begin{equation}\label{ineq:optimal}
\|u-Pu\|\leq d_n(A)\|\partial^r u\|,
\end{equation}
and $d_n(A)$ is the least possible constant over all $n$-dimensional subspaces $\mathcal{X}$.
\end{example}

Now, let $B$ be the unit ball in $L^2$, then in  \refcite{Pinkus:85,Melkman:78,Floater:2017,Floater:2018}, subsets $A$ of the form
\begin{equation*}%\label{eq:Hkernel}
 A = K(B) = \{ K f : \|f\| \le 1 \},
\end{equation*}
for various different integral operators $K$, are considered.
Observe that for such an $A$ we have the equality
\begin{equation}\label{eq:E1}
 E(A,\mathcal{X}) = \sup_{\|f\| \le 1} \| (I-P) K f \|= \|(I-P)K\|,
\end{equation}
where $\|(I-P)K\|$ is the $L^2$-operator norm of $(I-P)K$ which was used in Section~\ref{sec:error}.

The operator $K^\ast K$, being self-adjoint and positive semi-definite,
has eigenvalues
\begin{equation}\label{eq:lambda}
\lambda_1 \ge \lambda_2 \ge \cdots \ge \lambda_j \ge \cdots \ge 0,
\end{equation}
and corresponding orthogonal eigenfunctions
\begin{equation}\label{eq:phi}
 K^\ast K \phi_j = \lambda_j \phi_j, \quad j=1,2,\ldots.
\end{equation}
If we further define $\psi_j := K \phi_j$, then
\begin{equation}\label{eq:psi}
 K K^\ast \psi_j = \lambda_j \psi_j, \quad j=1,2,\ldots,
\end{equation}
and the $\psi_j$ are also orthogonal.
The square roots of the $\lambda_j$ are known as
the singular values of $K$ (or $K^*$).

\begin{example}
Let $K$ be the integral operator studied in Section~\ref{sec:per}. Recall that it satisfies $K^*=-K$, and so $KK^*=K^*K=-K^2$. Since the kernel of $K$ is the Green's function to problem \eqref{bvp:per}, one can show that the kernel of $-K^2$ is the Green's function to problem \eqref{eq:per}. Thus the $\lambda_j$ in \eqref{eq:lambda} satisfy $\lambda_j=1/\nu_j$ for $\nu_j$ as in \eqref{eq:eigvper}, $j=1,2,\ldots$. Moreover, the eigenfunctions $\psi_j$ in \eqref{eq:psi} are (up to a constant) equal to the $\psi_j$ in \eqref{eq:eigper} for $j=1,2,\ldots$. Note that $\nu_0=\lambda_\infty=0$.
\end{example}

From \eqref{eq:E1} we find that
\begin{align*}
 E(A,\mathcal{X}) &= \|(I-P)K\| = \|K^*(I-P)\| \\
  &= \sup_{\|f\| \le 1} (K^*(I-P)f,K^*(I-P)f)^{1/2}\\
  &= \sup_{\substack{\|f\| \le 1 \\ f \perp \mathcal{X}}}(K^*f,K^*f)^{1/2} =  \sup_{\substack{\|f\| \le 1 \\ f \perp \mathcal{X}}}(KK^*f,f)^{1/2},
\end{align*}
and by taking the infimum of the latter expression over all
$n$-dimensional subspaces $\mathcal{X}$ of $L^2$ we arrive at the following result (see p.~6 in \refcite{Pinkus:85}).
\begin{theorem}\label{thm:pinkus}
$d_n(A) =\lambda_{n+1}^{1/2}$, and the space $[\psi_1,\ldots,\psi_n]$
is optimal for $d_n(A)$.
\end{theorem}

\begin{example}\label{ex:per_nwidths}
Similar to \refcite{Floater:per}, we define $A^r_{\per} := \{u\in H^r_\per: \|\partial^r u\|\le 1\}$. It can be written as
$A^r_{\per} = \mathcal{P}_0\oplus K^r(B)$, where $K$ is the integral operator in Section~\ref{sec:per}.
Using Theorem~\ref{thm:pinkus} together with \eqref{eq:eigvper} and \eqref{eq:eigper}, we see that the $n$-widths in the periodic case, on the interval $(0,1)$, are given by
$$d_{2\nhalf-1}(A^r_\per)=d_{2\nhalf}(A^r_\per)=\Big(\frac{1}{2\pi \nhalf}\Big)^r,$$
and that the space of eigenfunctions in \eqref{pereig} is optimal for $d_{2\nhalf-1}(A^r_\per)$. Moreover, since the $(2\nhalf-1)$-width is equal to the $2\nhalf$-width, the $(2\nhalf-1)$-dimensional space in \eqref{pereig} is also optimal for $d_{2\nhalf}(A^r_\per)$.

If we let $\bftau$ be the uniform knot sequence with $h=1/(2\nhalf)$, then $\dim \mathcal{S}_{p,\bftau,\per}=2\nhalf$, and Theorem~\ref{thm:per} gives an alternative proof of Theorem~2 in \refcite{Floater:per}, that $\mathcal{S}_{p,\bftau,\per}$ is in this case optimal for $d_{2\nhalf}(A^r_\per)$ for all $p\geq r-1$.
It was shown in \refcite{Floater:per} that there is no knot sequence $\bftau$ such that $\mathcal{S}_{p,\bftau,\per}$ is optimal for $d_{2\nhalf-1}(A^r_\per)$. However,
if $h=1/(2\nhalf-1)$ then we obtain from Theorem~\ref{thm:per} that
\begin{equation*}
\|u-P_pu\|\le \Big(\frac{1}{\pi(2\nhalf-1)}\Big)^r\|\partial^r u\|,
\end{equation*}
and so it follows that periodic splines of dimension $2\nhalf-1$ on uniform knot sequences are asymptotically optimal as $\nhalf$ increases, i.e.,
$$\frac{1/(\pi(2\nhalf-1))^r}{d_{2\nhalf-1}(A^r_{\per})}\xrightarrow[\nhalf\to\infty]{} 1.$$
\end{example}

\section{Sharpness of error bounds}\label{sec:sharp}
In this section we discuss the sharpness of the error estimates obtained in this paper. To simplify we only consider the estimates for the $L^2$-projection. We call an error estimate \emph{sharp} if it is of the form \eqref{ineq:sharp} for some approximation $Pu$. If, additionally, the considered subspace is optimal, and the error estimate achieves the $n$-width (i.e., is of the form \eqref{ineq:optimal}), we refer to the error estimate as \emph{optimal}.

As we discussed in the previous section, in the periodic case the error estimate in Theorem~\ref{thm:per} achieves the $n$-width for all the optimal even-dimensional periodic spline spaces, and so the estimate is both sharp and optimal for these spaces.
For non-optimal, periodic spline spaces $\mathcal{S}_{p,\bftau,\per}$ it is unknown whether the estimate in Theorem~\ref{thm:per} is sharp. However, the $n$-widths in Example~\ref{ex:per_nwidths} provide a strict lower bound on $E(A^r_{\per},\mathcal{S}_{p,\bftau,\per})$, for $n=\dim \mathcal{S}_{p,\bftau,\per}$, that is ``very close'' to the upper bound of Theorem~\ref{thm:per}. Thus, the error estimate in Theorem~\ref{thm:per} is either sharp or very close to sharp in all cases.

The error estimate in Theorem~\ref{thm:one}, on the other hand, is optimal only in the simplest case given by the Poincar\'e inequality \eqref{ineq:deg0} for $r=1$, $p=0$ and with $\bftau$ being the uniform knot sequence.
For $r=1$ and $p\geq 1$ one can use a theorem of Karlovitz \cite{Karlovitz:76}, in a similar way to Section~6 of \refcite{Floater:per}, to show that the spline spaces $\mathcal{S}_{p,\bftau}$ are not optimal for the function class $\{u\in H^1: \|u'\|\le 1\}$ for any $\bftau$. For $p\geq 1$ the optimal spline spaces in this case are defined by choosing a degree-dependent knot sequence $\bftau$ and imposing certain boundary conditions on the space $\mathcal{S}_{p,\bftau}$ \cite{Floater:2017,Floater:2018}. These optimal spaces will be discussed further in Section~\ref{sec:optS}. % where we use the notation $\mathcal{S}_{p,1}$ to refer to them.
Turning back to the spaces $\mathcal{S}_{p,\bftau}$, one can use the $n$-width to find that for a fixed degree $p$, and for uniform knot sequences $\bftau$, the error estimate in Theorem~\ref{thm:one} is asymptotically optimal as $n=\dim \mathcal{S}_{p,\bftau}$ increases. In other words, we have for $A=\{u\in H^1: \|u'\|\le 1\}$, that 
$(h/\pi)/d_n(A)\to 1$ as $n\to\infty$,
since it is known that $d_n(A)=(b-a)/(n\pi)$ on the interval $(a,b)$ \cite{Kolmogorov:36,Floater:2018}.

For $r>1$ it was shown in \refcite{Melkman:78} that there exists a non-uniform, $r$-dependent, knot sequence $\bfxi$ such that $\mathcal{S}_{p,\bfxi}$ is optimal for $\{u\in H^r: \|\partial^r u\|\le 1\}$ for $p=r-1$. Since the error estimate in Theorem~\ref{thm:one} is minimized for uniform knot sequences, it cannot be sharp for this optimal spline space. In other words, for each degree $p>0$ there exist an $r$ ($r=p+1$) and a knot sequence $\bftau$ ($\bftau=\bfxi$) such that the constant in \eqref{eq:Sande} can be improved. It is then natural to ask whether it can be improved in all cases. Since the choice of $\bfxi$ is $r$-dependent one could expect that picking the knot sequence that gives the best possible approximation with respect to the $r$-th semi-norm could lead to worse approximation with respect to the semi-norms different from $r$. 

We therefore conjecture that for any given degree $p$ there exists a knot sequence $\bftau$ such that the factor $h/\pi$ in \eqref{eq:Sande} cannot be improved for all $r=1,\ldots,p+1$. In other words, we conjecture that for an arbitrary knot sequence~$\bftau$, Theorem~\ref{thm:one}, as stated, cannot be improved with a better approximation constant.

%The higher degree optimal spline spaces are again defined by taking certain knot sequences and boundary conditions that in general depends on both $r$ and $p$.

%%%%%%%%%%%%%%%%%%%%%%%%%%%%%%%%%%%%%
\section{General convergence to eigenfunctions}\label{sec:eig}
%%%%%%%%%%%%%%%%%%%%%%%%%%%%%%%%%%%%%
Motivated by the convergence results for the eigenfunctions in Corollary~\ref{cor:per},
in this section we will use results of \refcite{Floater:2018} to provide a general framework for obtaining convergence to the eigenfunctions of several differential operators. We then apply this framework to the Laplacian with different boundary conditions. This will lead us to different optimal spline spaces.

\subsection{General framework}
The starting point of our analysis is to look at the function classes $A^r$ and $A^r_*$ studied in \refcite{Floater:2018}.
For an arbitrary integral operator $K$, they can be defined as $A^1:=A:=K(B)$, $A^1_*:=A_*:=K^*(B)$ and
\begin{equation}\label{eq:Ar}
A^r:=K(A^{r-1}_*),\quad A^r_*:=K^*(A_{r-1}),
\end{equation}
for $r\geq 2 $, where we recall that $B$ is the unit ball in $L^2$. As stated in \refcite{Floater:2018}, it follows from Theorem \ref{thm:pinkus} that
\begin{equation*}
d_n(A^r_*)=d_n(A^r)=d_n(A)^r,
\end{equation*}
and the space $[\psi_1,\ldots,\psi_n]$ is optimal for $A^r$, and the space $[\phi_1,\ldots,\phi_n]$ is optimal for $A^r_*$, for all $r\geq 1$. Moreover, let $\mathcal{X}_0$ and $\mathcal{Y}_0$ be any finite dimensional subspaces of $L^2$ and define the subspaces $\mathcal{X}_p$ and $\mathcal{Y}_p$ in an analogous way to \eqref{eq:Ar}, by
\begin{equation*}%\label{eq:prop}
\mathcal{X}_p:=K(\mathcal{Y}_{p-1}), \quad \mathcal{Y}_p:=K^*(\mathcal{X}_{p-1}),
\end{equation*}
for $p\geq 1$.
It was then shown in Theorem~4 of \refcite{Floater:2018} that if $\mathcal{X}_0$ is optimal for the $n$-width of $A$ and $\mathcal{Y}_0$ is optimal for the $n$-width of $A_*$ then, for $r\geq 1$,
\begin{itemize}
 \item the subspaces $\mathcal{X}_p$ are optimal for the $n$-width of $A^{r}$, and
 \item the subspaces $\mathcal{Y}_p$ are optimal for the $n$-width of $A^{r}_*$,
\end{itemize}
for all $p\geq r-1$. Using this we can prove the following.
\begin{theorem}\label{thm:eig}
Suppose $\mathcal{X}_0$ is optimal for the $n$-width of $A$ and $\mathcal{Y}_0$ is optimal for the $n$-width of $A_*$. Let $P_p$ be the $L^2$-projection onto $\mathcal{X}_p$ and $\Pi_p$ be the $L^2$-projection onto $\mathcal{Y}_p$. Then, if there exists an index $\jmax\in\{1,2,\ldots,n\}$ such that $\lambda_\jmax>\lambda_{n+1}$ in \eqref{eq:lambda}, we have
\begin{align*}
\frac{\|(I-P_p)\psi_j\|}{\|\psi_j\|}  \xrightarrow[p\to\infty]{} 0,\quad
\frac{\|(I-\Pi_p)\phi_j\|}{\|\phi_j\|} \xrightarrow[p\to\infty]{} 0,
\end{align*}
for all $j=1,2,\ldots,\jmax$.
\end{theorem}
\begin{proof}
The two cases are analogous and so we only look at the $\mathcal{X}_p$. It follows from Theorem~4 in \refcite{Floater:2018} that $\mathcal{X}_p$ is optimal for $A^{p+1}$, and so by Theorem \ref{thm:pinkus},
\begin{equation*}
E(A^{p+1},\mathcal{X}_p)=d_n(A^{p+1}) =\lambda_{n+1}^{(p+1)/2}.
\end{equation*}
Equivalently, we have for all $f\in L^2$,
\begin{equation}\label{ineq:eig}
\begin{aligned}
\|(I-P_p)(KK^*)^{i}f\| &\leq \lambda_{n+1}^i\|f\|,  &&p=2i-1,\\
\|(I-P_p)(KK^*)^{i}Kf\| &\leq \lambda_{n+1}^{i+1/2}\|f\|,  &&p=2i.\\
\end{aligned}
\end{equation}
First, consider $p=2i$. Let $f=\phi_j$ for some $j=1,\ldots,\jmax$. Then,
\begin{equation*}
\|(I-P_p)(KK^*)^{i}K\phi_j\| = \|(I-P_p)\lambda_j^{i}\psi_j\| = \lambda_j^{i}\|(I-P_p)\psi_j\|,
\end{equation*}
and from \eqref{ineq:eig},
\begin{equation*}
\|(I-P_p)\psi_j\|\leq \Big(\frac{\lambda_{n+1}}{\lambda_j}\Big)^{i}(\lambda_{n+1})^{1/2}\|\phi_j\|.
\end{equation*}
Now, using $\psi_j=K\phi_j$, we have $\|\psi_j\|=(K^*K\phi_j,\phi_j)^{1/2}=\lambda_j^{1/2}\|\phi_j\|$, and
\begin{equation*}
\frac{\|(I-P_p)\psi_j\|}{\|\psi_j\|}\leq \Big(\frac{\lambda_{n+1}}{\lambda_j}\Big)^{i+1/2}\xrightarrow[p\to\infty]{} 0,
\end{equation*}
since $\lambda_j\geq \lambda_\jmax>\lambda_{n+1}$, and $p=2i$.

Next, consider $p=2i-1$. In this case we let $f=\psi_j$ for some $j=1,\ldots,\jmax$. Then, by a similar argument,
%\begin{equation*}
%\|(I-P_n^d)(KK^*)^{i}\psi_j\|= \|(I-P_n^d)\lambda_j^{2i}\psi_j\|= \lambda_j^{i}\|(I-P_n^d)\psi_j\|,
%\end{equation*}
%and so,
\begin{equation*}
\frac{\|(I-P_p)\psi_j\|}{\|\psi_j\|}\leq \Big(\frac{\lambda_{n+1}}{\lambda_j}\Big)^i \xrightarrow[p\to\infty]{} 0.
\end{equation*}
\end{proof}

\begin{remark}\label{rem:ritz}
The above result is only proved for $L^2$-projections and so it is not sufficient to conclude convergence of eigenfunctions and eigenvalues in Galerkin eigenvalue problems. However, a more careful analysis of the optimality results in \refcite{Floater:2018}, similar to the arguments of Section~\ref{sec:error}, could be used to show that Theorem~\ref{thm:eig} is true also when $P_p$ and $\Pi_p$ are replaced by certain Ritz-type projections. This is an interesting topic of further research.
\end{remark}

\subsection{Optimal spline spaces}\label{sec:optS}
Again, we consider the interval $(a,b)=(0,1)$, and 
we define 
the following subspaces of $\mathcal{S}_{p,\bftau}$ with certain derivatives vanishing at the boundary:
\begin{equation*} %\label{eq:allS}
\begin{aligned}
\mathcal{S}_{p,\bftau,0} &:= \{s\in \mathcal{S}_{p,\bftau} :\, \partial^\alpha s(0)=\partial^\alpha s(1)=0,\ \ 0\leq \alpha\leq p,\ \ \alpha \text{ even}\}, \\
 \mathcal{S}_{p,\bftau,1} &:= \{s\in \mathcal{S}_{p,\bftau} :\, \partial^\alpha s(0)=\partial^\alpha s(1)=0,\ \ 0\leq \alpha\leq p,\ \ \alpha \text{ odd}\}, \\
 \mathcal{S}_{p,\bftau,2}&:= \{s\in \mathcal{S}_{p,\bftau} :\, \partial^{\alpha_0} s(0)=\partial^{\alpha_1} s(1)=0,\ \ 0\leq \alpha_0,\alpha_1\leq p, \ \ \alpha_0 \text{ even}, \ \ \alpha_1 \text{ odd}\}.
\end{aligned}
\end{equation*}

For the special (degree-dependent) knot sequences $\bftau_{p,i}$, $i=0,1,2$, where
\begin{equation*} %\label{eq:alltau}
\begin{aligned}
\bftau_{p,0} &:= \begin{cases}
            (0,\frac{1}{n+1},\frac{2}{n+1},\ldots,\frac{n}{n+1},1),\quad\ &p \text{ odd},\\
            (0,\frac{1/2}{n+1},\frac{3/2}{n+1},\ldots,\frac{n+1/2}{n+1},1),\quad\ &p \text{ even},
            \end{cases}\\
\bftau_{p,1} &:= \begin{cases}
            (0,\frac{1/2}{n},\frac{3/2}{n},\ldots,\frac{n-1/2}{n},1),\qquad &p \text{ odd},\\
            (0,\frac{1}{n},\frac{2}{n},\ldots,\frac{n-1}{n},1),\qquad &p \text{ even},
            \end{cases}\\
\bftau_{p,2} &:= \begin{cases}
            (0,\frac{1}{2n+1},\frac{3}{2n+1},\ldots,\frac{2n-1}{2n+1},1),\quad &p \text{ even},\\
            (0,\frac{2}{2n+1},\frac{4}{2n+1},\ldots,\frac{2n}{2n+1},1),\quad &p \text{ odd},
            \end{cases}
\end{aligned}
\end{equation*}
it was shown in \refcite{Floater:2018} that the spline spaces $\mathcal{S}_{p,i}:=\mathcal{S}_{p,\bftau_{p,i},i}$, $i=0,1,2$ are optimal, respectively, for the function classes
\begin{equation*}
\begin{aligned}
A^r_0&:=\{u\in H^r:\, \|\partial^r u\|\leq 1,\ \ \partial^\alpha u(0)=\partial^\alpha u(1)=0,\ \ 0\leq \alpha<r,\ \ \alpha\text{ even}\},
\\
A^r_1&:=\{u\in H^r:\, \|\partial^r u\|\leq 1,\ \ \partial^\alpha u(0)=\partial^\alpha u(1)=0,\ \ 0\leq \alpha<r,\ \ \alpha\text{ odd}\},
\\
A^r_2&:=\{u\in H^r:\, \|\partial^r u\|\leq 1,\ \ \partial^{\alpha_0} u(0)=\partial^{\alpha_1} u(1)=0,\ \ 0\leq \alpha_0,\alpha_1<r,\\
&\hspace*{8cm} \alpha_0 \text{ even}, \ \ \alpha_1 \text{ odd}\},
\end{aligned}
\end{equation*}
for all $p\geq r-1$. 
It was further shown that the function classes $A^r_0$ and $A^r_2$ are examples of the function classes $A^r_*$ and $A^r$ in \eqref{eq:Ar}, while $A^r_1$ is of the form $\mathcal{P}_0\oplus A^r$.

The optimal $n$-dimensional space of eigenfunctions for $A^r_i$, $i=0,1,2$, consists of the first $n$ eigenfunctions of the Laplacian satisfying the following zero boundary conditions, respectively,  
\begin{equation*} %\label{eq:BC-Laplace}
\begin{gathered}
u(0)=u(1)=0, \\
u'(0)=u'(1)=0,\\
u(0)=u'(1)=0;
\end{gathered}
\end{equation*}
in other words, the functions
\begin{equation*} %\label{eq:eig-Laplace}
\begin{gathered}
\{\sin(\pi x), \sin(2\pi x), \ldots,\sin(n\pi x)\},\\
\{1,\cos(\pi x), \ldots,\cos((n-1)\pi x)\},\\
\{\sin((1/2)\pi x), \sin((3/2)\pi x), \ldots, \sin((n-1/2)\pi x)\}.
\end{gathered}
\end{equation*}
Since the eigenvalues are in these cases strictly decreasing, it then follows from Theorem~\ref{thm:eig} 
that the $n$-dimensional spline spaces $\mathcal{S}_{p,i}$, converge, respectively, to the first $n$ sines, cosines and shifted sines.
We make this statement more precise in the next corollary.

\begin{corollary}\label{cor:eig}
 If $P_{p,i}$ denotes the $L^2$-projection onto the $n$-dimensional spline spaces $\mathcal{S}_{p,i}$, then
\begin{align*}
\|(I-P_{p,0})\sin(j\pi\cdot)\|  &\xrightarrow[p\to\infty]{} 0,\\
\|(I-P_{p,1})\cos((j-1)\pi\cdot)\| &\xrightarrow[p\to\infty]{} 0,\\
\|(I-P_{p,2})\sin((j-1/2)\pi\cdot)\| &\xrightarrow[p\to\infty]{} 0,
\end{align*}
for $j=1,\ldots,n$. Here $\cos(0)$ refers to the constant function $1$.
\end{corollary}
\begin{proof}
As shown in \refcite{Floater:2018}, the optimal spline spaces $\mathcal{S}_{p,i}$, $i=0,1,2$, are examples of the spaces $\mathcal{X}_p$ and $\mathcal{Y}_p$ for different choices of $K$. Since in all cases we have $\lambda_n>\lambda_{n+1}$, Theorem~\ref{thm:eig} concludes the proof. Note that $\mathcal{S}_{p,1}$ is in fact of the form $\mathcal{P}_0\oplus \mathcal{X}_{p}$, and so the constant functions are in the space $\mathcal{S}_{p,1}$ for all~$p$.
\end{proof}

\begin{remark}
%We remark that the 
The above corollary can be generalized to the tensor-product case by using Theorem~8 in \refcite{Bressan:preprint}.
Specifically, one can use Theorem~8 in \refcite{Bressan:preprint} to obtain an error estimate for $\mathcal{S}_{p,0}\otimes\mathcal{S}_{p,0}$ in an analogous way to Section~\ref{subsec:tens}. Then, similarly to how we proved Corollary \ref{cor:per}, one can plug the eigenfunctions of the Laplacian with zero boundary conditions on the square $(0,1)^2$ into this error estimate and, since these eigenfunctions are just tensor products of the above sines (the eigenfunctions of the 1D Laplacian with zero boundary conditions), show convergence of $\mathcal{S}_{p,0}\otimes\mathcal{S}_{p,0}$ to the first $n^2$ tensor-product eigenfunctions as $p\to\infty$.
Similar arguments apply to $\mathcal{S}_{p,1}\otimes\mathcal{S}_{p,1}$ and $\mathcal{S}_{p,2}\otimes\mathcal{S}_{p,2}$.
\end{remark}

\section{Error estimates for reduced spline spaces}\label{sec:reduced}
In this section we focus on error estimates for the subspaces $\mathcal{S}_{p,\bftau,1}\subset \mathcal{S}_{p,\bftau}$ defined in Section~\ref{sec:optS} and for the following variations
\begin{equation*} %\label{eq:redS}
 \widetilde{\mathcal{S}}_{p,\bftau} := 
\{s\in \mathcal{S}_{p,\bftau} :\, \partial^\alpha s(0)=\partial^\alpha s(1)=0,\ \ 0\leq \alpha< p,\ \ \alpha \text{ odd}\}.
\end{equation*}
For uniform knot sequences, the latter are the ``reduced spline spaces'' investigated in \refcite{Takacs:2016} (see Definition 5.1 of \refcite{Takacs:2016}).
For even degrees $p$, the spaces $\mathcal{S}_{p,\bftau,1}$ are exactly $\widetilde{\mathcal{S}}_{p,\bftau}$.
If we further remove the two last boundary conditions $\partial^p s(0)=\partial^p s(1)=0$ in $\mathcal{S}_{p,\bftau,1}$ for $p$ odd (and thus increase the dimension of $\mathcal{S}_{p,\bftau,1}$ by two in this case), then we again obtain a space $\widetilde{\mathcal{S}}_{p,\bftau}$. 

Using the Poincar\'e inequality \eqref{ineq:deg0} and Lemma~1 in \refcite{Floater:2017} one can prove that the case $r=1$ of Theorem~\ref{thm:one} holds for these reduced spline spaces $\widetilde{\mathcal{S}}_{p,\bftau}$. However, if we do not remove the two last boundary conditions $\partial^p s(0)=\partial^p s(1)=0$ in $\mathcal{S}_{p,\bftau,1}$ for $p$ odd, then the obtained error estimate would, for some knot sequences $\bftau$, be worse than Theorem~\ref{thm:one} by a factor of~$2$. This is the content of Theorems~\ref{thm:Sp1} and~\ref{thm:reduced}. We start by proving the following intermediate result.
%\end{remark}

\begin{lemma}\label{lem:Sp01}
For any knot sequence $\bftau$, let $\widehat h:=\max\{2h_0,h_1,h_2,\ldots, h_{\nknots-1},2h_{\nknots}\}$. If $P_0$ denotes the $L^2$-projection onto the spline space $\mathcal{S}_{0,\bftau,0}$, then for any function $u\in H^1_0$ we have
\begin{equation*}
\|u-P_0u\|\leq \frac{\widehat h}{\pi}\|u'\|.
\end{equation*}
\end{lemma}
\begin{proof}
Recall that any element in $\mathcal{S}_{0,\bftau,0}$ is identically zero on the first and last knot intervals and  piecewise constant in the interior. The result then follows by the same argument as in \eqref{ineq:deg0proof}. We apply the Poincar\'e inequality in \eqref{ineq:Poinc} for the interior knot intervals. For the first and last knot intervals we apply the inequality
\begin{equation*}
\|u\|\leq \frac{2(b-a)}{\pi}\|u'\|,
\end{equation*}
that holds for all $u\in H^1$ on an interval $(a,b)$ satisfying either $u(a)=0$ or $u(b)=0$ (see, e.g., the case $n=0$ and $i=2$ of Theorem~1 in \refcite{Floater:2018}).
\end{proof}

In the proof of the next theorem we apply Lemma \ref{lem:1simpl} using an integral operator that integrates the spline space $\mathcal{S}_{\pnew,\bftau,1}$ twice. As a consequence, the $\prec$ in Lemma \ref{lem:1simpl} will, in this case, not correspond to the degree $\pnew$ of $\mathcal{S}_{\pnew,\bftau,1}$. 

\begin{theorem}\label{thm:Sp1}
For any knot sequence $\bftau$, let $h$ denote its maximum knot distance and let $\widehat h:=\max\{2h_0,h_1,h_2,\ldots, h_{\nknots-1},2h_{\nknots}\}$. If $P_{\pnew}$ denotes the $L^2$-projection onto the spline space $\mathcal{S}_{\pnew,\bftau,1}$, then for any function $u\in H^1$ we have
\begin{equation*}
\begin{aligned}
\|u-P_{\pnew}u\|&\leq \frac{h}{\pi}\|u'\|, &&\pnew \text{ even},\\
\|u-P_{\pnew}u\|&\leq \frac{\widehat h}{\pi}\|u'\|, &&\pnew \text{ odd}.
\end{aligned}
\end{equation*}
\end{theorem}
\begin{proof}
Let $\Pi$ denote the $L^2$-projection onto $\mathcal{P}_0$ and define the integral operator $K_1:=(I-\Pi)K$, where $K$ is the integral operator in \eqref{eq:Kint}. From \eqref{eq:Hr} it follows that $H^1=\mathcal{P}_0\oplus K_1(L^2).$ Furthermore, as shown in  \refcite{Floater:2017}, the kernel of the self-adjoint operator $K_1K_1^*$ is the Green's function to the boundary value problem
\begin{equation}\label{bvp:neumann}
-u''(x)=f(x),\quad x\in (0,1),\quad u'(0)=u'(1)=0,\quad u,f\perp 1,
\end{equation}
and we have the orthogonal decomposition (see \refcite{Floater:2018})
\begin{equation}\label{eq:Sp1orth}
\mathcal{S}_{\pnew,\bftau,1} = \mathcal{P}_0\oplus K_1K_1^*(\mathcal{S}_{\pnew-2,\bftau,1}),\quad \pnew\geq 2.
\end{equation}

For $\pnew=0$ the result follows from the Poincar\'e inequality as shown in \eqref{ineq:deg0}. If $\pnew$ is an even number the result then follows from Lemma~\ref{lem:1simpl} with $K_1K_1^*$ playing the role of $K$ and with $\pnew=2\prec$.

Next, we consider the case $\pnew=1$. Using the definition of $K_1$ we know that $u\in H^1$ can be written as the orthogonal sum $u=c+K_1f$ for $c\in\mathcal{P}_0$ and $f\in L^2$. From \refcite{Floater:2018} we recall the decomposition
\begin{equation}\label{eq:Sp1orth2}
\mathcal{S}_{\pnew,\bftau,1} = \mathcal{P}_0\oplus K_1(\mathcal{S}_{\pnew-1,\bftau,0}),\quad \pnew\geq 1,
\end{equation}
and we define $\hat{P}_0$ to be the $L^2$-projection onto $\mathcal{S}_{0,\bftau,0}$. Using \eqref{eq:Sp1orth2} it follows that $K_1\hat{P}_0$ maps into the space $\mathcal{S}_{1,\bftau,1}$, and so
\begin{align*} 
\|u-P_1u\| &= \|K_1f-P_1K_1f\|\leq \|K_1f-K_1\hat{P}_0f\| \\
 &\leq \|K_1(I-\hat{P}_0)\|\,\|f\| = \|(I-\hat{P}_0)K_1^*\|\,\|f\|.
\end{align*}
Since $H^1_0=K_1^*(L^2)$ (see \refcite{Floater:2018}), we deduce from Lemma~\ref{lem:Sp01} that 
%$\|(I-\hat{P}_0)K_1^*\|\leq \widehat h/\pi$, 
$$\|(I-\hat{P}_0)K_1^*\|\leq \frac{\widehat h}{\pi},$$
which proves the case $\pnew=1$. 
If $\pnew$ is an odd number the result then follows from Lemma~\ref{lem:1simpl} and \eqref{eq:Sp1orth}.
\end{proof}

\begin{theorem}\label{thm:reduced}
For any knot sequence $\bftau$, 
%let $\widetilde{\mathcal{S}}_{\pnew,\bftau}$ denote the reduced spline space of \refcite{Takacs:2016}. Further, 
let $h$ be the maximum knot distance of $\bftau$, and let $P_{\pnew}$ denote the $L^2$-projection onto the spline space $\widetilde{\mathcal{S}}_{\pnew,\bftau}$. Then, for any function $u\in H^1$ we have
\begin{equation*}
\|u-P_{\pnew}u\|\leq \frac{h}{\pi}\|u'\|.
\end{equation*}
\end{theorem}
\begin{proof}
The case of $\pnew$ even is covered by Theorem \ref{thm:Sp1}, since $\widetilde{\mathcal{S}}_{\pnew,\bftau}=\mathcal{S}_{\pnew,\bftau,1}$ in this case. For $\pnew=1$ the result is the case $r=1$ of Theorem~\ref{thm:one}, since $\widetilde{\mathcal{S}}_{1,\bftau}=\mathcal{S}_{1,\bftau}$. We now consider odd degrees $\pnew>1$. Letting $K_1$ be the integral operator in the proof of Theorem \ref{thm:Sp1}, it follows from \eqref{bvp:neumann} that 
\begin{equation*}
\widetilde{\mathcal{S}}_{\pnew,\bftau} = \mathcal{P}_0\oplus K_1K_1^*(\widetilde{\mathcal{S}}_{\pnew-2,\bftau}),\quad \pnew\geq 2,
\end{equation*}
since the derivative of a spline is a spline on the same knot vector of one degree lower. Using Lemma~\ref{lem:1simpl} with $K_1K_1^*$ playing the role of $K$ we obtain the claimed result.
\end{proof}

\section{Inverse inequalities} \label{sec:inverse}
In this section we show that the spline spaces $\mathcal{S}_{p,\bftau,\per}$ (see Section~\ref{sec:per}), $\mathcal{S}_{p,\bftau,i}$, $i=0,1,2$ (see Section~\ref{sec:optS}), and $\widetilde{\mathcal{S}}_{p,\bftau}$ (see Section~\ref{sec:reduced}) all satisfy an inverse inequality for any knot sequence $\bftau$.
The proof of the following theorem is done by induction. The base case can be found in Theorem~3.91 of \refcite{Schwab:99} and the induction step in Theorem~6.1 of \refcite{Takacs:2016}, but only for the reduced spline spaces $\widetilde{\mathcal{S}}_{p,\bftau}$. The case $\mathcal{S}_{p,\bftau,0}$ was later shown in \refcite{Sogn:2018}. For the sake of completeness we give the full proof here in a general form.

\begin{theorem}\label{thm:inv}
For any knot sequence $\bftau$, let $\hmin$ denote its minimum knot distance.
%Let $\hmin$ be the minimal knot distance of a knot sequence $\bftau$.
For $p\geq 1$, assume $\mathcal{S}_p$ is any subspace of $\mathcal{S}_{p,\bftau}$ such that the boundary conditions
\begin{equation}\label{cond:inv}
 \partial^\alpha s(0)\partial^{\alpha-1} s(0)= \partial^\alpha s(1)\partial^{\alpha-1} s(1), \quad \alpha=1,\ldots,p-1
\end{equation}
are satisfied for all $s\in \mathcal{S}_p$. Then, the inverse inequality in \eqref{ineq:inv}
% \begin{equation}\label{ineq:inv}
% \|s'\|\leq \frac{2\sqrt{3}}{\hmin}\|s\|, \qquad s\in\mathcal{S}_p,
% \end{equation}
holds.
\end{theorem}
\begin{proof}
We first look at $p=1$. We will use a scaling argument. Let $\hat s$ be a linear function on the interval $[-1,1]$. Since it can be written as $\hat s(x)= a_0+a_1x$, we get
$$\|\hat s'\|^2=2a_1^2\leq3(2a_0^2+\frac{2}{3}a_1^2)=3\|\hat s\|^2.$$
By repeating this argument on each knot interval $I_j$, we have
% $$\|s'|_{I_j}\|_j\leq \frac{2\sqrt{3}}{h_j}\|s|_{I_j}\|_j,$$
% for all $s\in \mathcal{S}_{1,\bftau}$.
% Finally, we have
% $$\|s'\|^2=\sum_{j=0}^\nknots\|s'|_{I_j}\|_j^2\leq \sum_{j=0}^\nknots \Big(\frac{2\sqrt{3}}{h_j}\|s|_{I_j}\|_j\Big)^2\leq \Big(\frac{2\sqrt{3}}{\hmin}\|s\|\Big)^2.$$
$$\|s'\|_j\leq \frac{2\sqrt{3}}{h_j}\|s\|_j,$$
for all $s\in \mathcal{S}_{1,\bftau}$.
Finally, we arrive at
$$\|s'\|^2=\sum_{j=0}^\nknots\|s'\|_j^2\leq \sum_{j=0}^\nknots \Big(\frac{2\sqrt{3}}{h_j}\|s\|_j\Big)^2\leq \Big(\frac{2\sqrt{3}}{\hmin}\|s\|\Big)^2.$$

Next, we assume the result is true for $\mathcal{S}_{p-1}$ and consider the case of $\mathcal{S}_p$.
Using integration by parts and the Cauchy-Schwarz inequality we have
$$
\|s'\|^2=\int_0^1(s'(x))^2 d x = [s's]_0^1 -\int_0^1 s''(x)s(x) dx\leq \|s''\|\,\|s\|,
$$
where the boundary terms disappeared since $s\in\mathcal{S}_p$. Now, using the induction hypothesis together with the fact that $s'\in\mathcal{S}_{p-1}$ whenever $s\in\mathcal{S}_p$, we obtain
$$
\|s'\|^2\leq \|s''\|\,\|s\|\leq \frac{2\sqrt{3}}{\hmin}\|s'\|\,\|s\|,
$$
and the result follows.
\end{proof}

The spline spaces $\mathcal{S}_{p,\bftau,\per}$, $\mathcal{S}_{p,\bftau,i}$, $i=0,1,2$, and $\widetilde{\mathcal{S}}_{p,\bftau}$ all satisfy the boundary conditions in~\eqref{cond:inv}. %Hence, we conclude the following result.
This brings us to the following corollary.
\begin{corollary}
The spline spaces $\mathcal{S}_{p,\bftau,\per}$, $\mathcal{S}_{p,\bftau,i}$, $i=0,1,2$, and $\widetilde{\mathcal{S}}_{p,\bftau}$ satisfy the inverse inequality in~\eqref{ineq:inv}.
\end{corollary}
% \begin{proof}
% 
% \end{proof}

%%%%%%%%%%%%%%%%%%%%%%%%%%%%%%%%%%%%%%%%%%%
\section{Conclusions}\label{sec:conclusions}
%%%%%%%%%%%%%%%%%%%%%%%%%%%%%%%%%%%%%%%%%%%
In this paper we have introduced a general framework for deriving error estimates based on the theory of Kolmogorov $L^2$ $n$-widths and the representation of Sobolev spaces in terms of integral operators described by suitable kernels. By applying this framework we have obtained sharp (or close to sharp) error estimates for spline approximation, in both the periodic and the non-periodic case. These generalize and/or improve the results known in the literature. 
More precisely, for the important case of spline spaces of maximal smoothness, we have presented the following contributions:
\begin{itemize}
\item we have provided error estimates for the $L^2$-projection and Ritz projections of any function in $H^r$ for arbitrary knot sequences and with explicit constants;
\item focusing on the periodic case, we have 
used the error estimate for the Ritz projection to prove convergence of the Galerkin method, in $p$, to the eigenvalues and eigenfunctions of the Laplacian with periodic boundary conditions;
\item we have related the problem of spectral convergence to the theory of Kolmogorov $n$-widths and proved a general convergence result for optimal subspaces;
\item we have identified $n$-dimensional spline spaces, all satisfying an inverse inequality and all possessing optimal approximation order for function classes in $H^1$, that converge, in $p$, to the $n$ first eigenfunctions of the Laplacian with various boundary conditions.
\end{itemize}
Besides the direct theoretical interests of the presented results, we also see several practical consequences in the IGA context:
\begin{itemize}
  \item they can be a starting point for the theoretical understanding of the benefits of spline approximation under $k$-refinement by isogeometric discretization methods;
  \item they provide theoretical insights into the outperformance of smooth spline discretizations of eigenvalue problems, that has been numerically observed in the literature;
  \item they form a theoretical foundation for proving optimality, in $n$ and $p$, of geometric multigrid solvers for linear systems arising from (non-uniform) smooth spline discretizations.
\end{itemize}

\section*{Acknowledgements}
The authors are indebted to M.~S.~Floater for the numerous discussions which improved the content and the presentation of the paper.
This work was supported by the MIUR Excellence Department Project awarded to the Department of Mathematics, University of Rome Tor Vergata (CUP E83C18000100006) and received funding from the European Research Council under the European Union's Seventh Framework Programme (FP7/2007-2013) / ERC grant agreement 339643. 
C.~Manni and H.~Speleers are members of Gruppo Nazionale per il Calcolo Scientifico, Istituto Nazionale di Alta Matematica.

%%%%%%%%%%%%%%%%%%%%%%%%%%%%%%%%%%%%%%
\bibliographystyle{ws-m3as}
\bibliography{nwidths}
%%%%%%%%%%%%%%%%%%%%%%%%%%%%%%%%%%%%%

\end{document}